\documentclass[12p]{amsart}
\usepackage{amssymb}
\usepackage{amsmath}
\usepackage{amsfonts}
\usepackage{geometry}
\usepackage{graphicx}
\usepackage{mathrsfs,amssymb,latexsym}

\usepackage{hyperref}
\usepackage{cleveref}

\theoremstyle{plain}

\newtheorem{corollary}{Corollary}

\newtheorem{definition}{Definition}

\newtheorem{lemma}{Lemma}

\newtheorem{proposition}{Proposition}
\newtheorem{remark}{Remark}

\newtheorem{theorem}{Theorem}
\numberwithin{equation}{section}

\usepackage{cancel}

\begin{document}
\title{Sharp Global well-posedness and scattering for the radial nonlinear intercritical wave equation}
\date{\today}
\author{Benjamin Dodson}
\maketitle

\begin{abstract}
In this paper we prove global well-posedness and scattering for the defocusing, intercritical nonlinear wave equation in dimensions $d \geq 4$ with radial initial data. We prove this for sharp initial data.
\end{abstract}

\section{Introduction}
In this paper we prove global well-posedness and scattering for the defocusing, nonlinear wave equation
\begin{equation}\label{1.1}
u_{tt} - \Delta u + |u|^{p - 1} u = 0, \qquad u(0, x) = u_{0}, \qquad u_{t}(0, x) = u_{1}, \qquad u : \mathbb{R} \times \mathbb{R}^{d} \rightarrow \mathbb{R},
\end{equation}
with radially symmetric initial data in dimensions $d \geq 4$ and $\frac{4}{d - 1} < p - 1 < \frac{4}{d - 2}$. This continues an earlier study we began in \cite{dodson2018global2} and \cite{dodson2021global} when $d = 3$. Note that when $d = 3$, $3 < p < 5$. When $p = \frac{4}{d - 1}$, $(\ref{1.1})$ is a conformal wave equation. See \cite{dodson2018globalAPDE}, \cite{dodson2018global}, \cite{dodson2022global}, and \cite{dodson2023sharp} for global well-posedness and scattering for radially symmetric initial data for the conformal wave equation. See also \cite{miao2020global}. When $p - 1 = \frac{4}{d - 2}$, $(\ref{1.1})$ is the energy--critical nonlinear wave equation. See \cite{kenig2015lectures} and the references therein for the defocusing, nonlinear, energy--critical wave equation.\medskip

Equation $(\ref{1.1})$ possesses a scaling symmetry. Specifically, if $u$ solves $(\ref{1.1})$, then for any $\lambda > 0$,
\begin{equation}\label{1.4}
v(t, x) = \lambda^{\frac{2}{p - 1}} u(\lambda t, \lambda x),
\end{equation}
also solves $(\ref{1.1})$ with initial data
\begin{equation}\label{1.5}
v(0, x) = \lambda^{\frac{2}{p - 1}} u_{0}(\lambda x), \qquad v_{t}(0, x) = \lambda^{\frac{p + 1}{p - 1}} u_{1}(\lambda x).
\end{equation}
For
\begin{equation}\label{1.5.1}
s_{c} = \frac{d}{2} - \frac{2}{p - 1},
\end{equation}
the $\dot{H}^{s_{c}} \times \dot{H}^{s_{c} - 1}$ norm of the initial data is conserved. That is, $\| v(0, x) \|_{\dot{H}^{s_{c}}} = \| u_{0} \|_{\dot{H}^{s_{c}}}$ and $\| v_{t}(0, x) \|_{\dot{H}^{s_{c} - 1}} = \| u_{1} \|_{\dot{H}^{s_{c} - 1}}$ for any $\lambda > 0$. This fact was well-exploited by \cite{lindblad1995existence} to prove ill-posedness for the cubic wave equation with initial data in $\dot{H}^{s} \times \dot{H}^{s - 1}$ for $s < s_{c} = \frac{1}{2}$. See also \cite{christ2003asymptotics}. The idea is as follows. As a general rule for solving equations of the form $(\ref{1.1})$, small data is easier than large data and short times are easier than long times. If $s > s_{c}$, $(\ref{1.4})$ allows one to choose between solving a problem with small data for long times or solving a problem with large data for short times. However, when $s < s_{c}$, $(\ref{1.4})$ shows that small data for short times is as hard as large data for long times, which yields ill-posedness. See Section $3.1$ of \cite{tao2006nonlinear} for more information on this heuristic.\medskip

Therefore, the result proved here is sharp.

\begin{theorem}\label{t1.1}
The initial value problem $(\ref{1.1})$ is globally well-posed and scattering for any radially symmetric initial data $u_{0} \in \dot{H}^{s_{c}}(\mathbb{R}^{d})$ and $u_{1} \in \dot{H}^{s_{c} - 1}(\mathbb{R}^{d})$. Moreover, there exists a function $C(d, \| u_{0} \|_{\dot{H}^{s_{c}}}, \| u_{1} \|_{\dot{H}^{s_{c} - 1}})$,
\begin{equation}\label{1.2}
C : \mathbb{Z}_{\geq 4} \times [0, \infty) \times [0, \infty) \rightarrow [0, \infty),
\end{equation}
such that if $u$ is the solution to $(\ref{1.1})$ with radially symmetric initial data in $u_{0}$, $u_{1}$,
\begin{equation}\label{1.3}
\| u \|_{L_{t,x}^{\frac{(d + 1)(p - 1)}{2}}(\mathbb{R} \times \mathbb{R}^{d})} \leq C(d, \| u_{0} \|_{\dot{H}^{s_{c}}}, \| u_{1} \|_{\dot{H}^{s_{c} - 1}}).
\end{equation}
\end{theorem}

\begin{definition}[Global well-posedness and scattering]
Here we use the standard definitions of global well-posedness and scattering. Specifically, global well-posedness means that a solution to $(\ref{1.1})$ exists, the solution is unique, and the solution depends continuously on the initial data. In this paper, a solution means a solution $u \in L_{t, loc}^{\frac{(d + 1)(p - 1)}{2}} L_{x}^{\frac{(d + 1)(p - 1)}{2}}$ which satisfies Duhamel's principle,
\begin{equation}\label{1.3.1}
u(t) = \cos(t \sqrt{-\Delta}) u_{0} + \frac{\sin(t \sqrt{-\Delta})}{\sqrt{-\Delta}} u_{1} - \int_{0}^{t} \frac{\sin((t - \tau) \sqrt{-\Delta}}{\sqrt{-\Delta}} |u(\tau)|^{p - 1} u(\tau) d\tau.
\end{equation}
By scattering, we mean that there exist $u_{0}^{+}, u_{0}^{-} \in \dot{H}^{s_{c}}$, $u_{1}^{+}, u_{1}^{-} \in \dot{H}^{s_{c} - 1}$, such that,
\begin{equation}\label{1.3.2}
\lim_{t \rightarrow \infty} \| u(t) - \cos(t \sqrt{-\Delta}) u_{0}^{+} - \frac{\sin(t \sqrt{-\Delta})}{\sqrt{-\Delta}} u_{1}^{+} \|_{\dot{H}^{s_{c}}} = 0,
\end{equation}
\begin{equation}
\lim_{t \rightarrow \infty} \| \partial_{t}(u(t) - \cos(t \sqrt{-\Delta}) u_{0}^{+} - \frac{\sin(t \sqrt{-\Delta})}{\sqrt{-\Delta}} u_{1}^{+}) \|_{\dot{H}^{s_{c} - 1}} = 0,
\end{equation}
\begin{equation}\label{1.3.3}
\lim_{t \rightarrow -\infty} \| u(t) - \cos(t \sqrt{-\Delta}) u_{0}^{-} - \frac{\sin(t \sqrt{-\Delta})}{\sqrt{-\Delta}} u_{1}^{-} \|_{\dot{H}^{s_{c}}} = 0,
\end{equation}
and
\begin{equation}\label{1.3.4}
\lim_{t \rightarrow -\infty} \| \partial_{t}(u(t) - \cos(t \sqrt{-\Delta}) u_{0}^{-} - \frac{\sin(t \sqrt{-\Delta})}{\sqrt{-\Delta}} u_{1}^{-}) \|_{\dot{H}^{s_{c} - 1}} = 0.
\end{equation}
\end{definition}

\subsection{Outline of previous results}
It has been known for a long time that global well-posedness and scattering hold for $(\ref{1.1})$ with initial data in $\dot{H}^{1} \times L^{2}$ that decays sufficiently fast as $|x| \rightarrow \infty$, see \cite{strauss1981nonlinear} and \cite{strauss1968decay}. Observe that such initial data has the conserved energy,
\begin{equation}\label{1.7}
E(u) = \frac{1}{2} \int |\nabla u(t, x)|^{2} dx + \frac{1}{2} \int |u_{t}(t, x)|^{2} dx + \frac{1}{p + 1} \int |u(t,x)|^{p + 1} dx.
\end{equation}
Conservation of the conformal energy gives scattering, which will be shown in section three.\medskip

For initial data in $\dot{H}^{1} \cap \dot{H}^{s_{c}} \times L^{2} \cap \dot{H}^{s_{c} - 1}$, global well-posedness follows easily from $(\ref{1.7})$ and the local well-posedness result of \cite{lindblad1995existence}.
\begin{theorem}\label{t1.2}
The initial value problem $(\ref{1.1})$ is locally well-posed on some interval $(-T, T)$ for any $(u_{0}, u_{1}) \in \dot{H}^{s_{c}} \times \dot{H}^{s_{c} - 1}$, where $T = T(u_{0}, u_{1})$. Global well-posedness and scattering hold for small initial data.

Moreover, the solution satisfies
\begin{equation}\label{1.8}
u \in L_{t}^{\infty} \dot{H}^{s_{c}}((-T, T) \times \mathbb{R}^{d}), \qquad u_{t} \in L_{t}^{\infty} \dot{H}^{s_{c} - 1}((-T, T) \times \mathbb{R}^{d}), \qquad u \in L_{t, loc}^{\frac{(d + 1)(p - 1)}{2}} L_{x}^{\frac{(d + 1)(p - 1)}{2}}((-T, T) \times \mathbb{R}^{d}).
\end{equation}

Furthermore, if the solution only exists on an interval $[0, T_{+})$ for some $T_{+} < \infty$, then
\begin{equation}\label{1.9}
\lim_{T \nearrow T_{+}} \| u \|_{L_{t,x}^{\frac{(d + 1)(p - 1)}{2}}([0, T] \times \mathbb{R}^{d})} = \infty.
\end{equation}
By time reversal symmetry, the analogous result holds on $(-T_{-}, 0]$.
\end{theorem}
\noindent Therefore, $(\ref{1.1})$ has a local solution, and by $(\ref{1.7})$, blowup cannot occur in finite time.\medskip

Now suppose we take initial data of the form
\begin{equation}\label{1.6}
u_{0} = c_{1} e^{-|x|^{2}} + c_{2} \lambda_{1}^{\frac{2}{p - 1}} e^{-|\lambda_{1} x|^{2}} + c_{3} \lambda_{2}^{\frac{2}{p - 1}} e^{-|\lambda_{2} x|^{2}}, \qquad u_{1} = 0, \qquad \lambda_{1} \ll 1 \ll \lambda_{2}.
\end{equation}
The results of \cite{strauss1981nonlinear} and \cite{strauss1968decay} imply scattering for $(\ref{1.1})$ with initial data given by $(\ref{1.6})$ with $c_{1} = c_{3} = 0$. Furthermore, when $c_{1} = c_{3} = 0$, the scattering size, the $\| u \|_{L_{t,x}^{\frac{(d + 1)(p - 1)}{2}}}$ norm, is polynomially dependent on $c_{2}$. The scaling symmetry in $(\ref{1.4})$ implies the same estimates hold when $c_{2} = c_{3} = 0$ or $c_{1} = c_{2} = 0$.

For generic $c_{1}$, $c_{2}$, and $c_{3}$, the $L_{t,x}^{\frac{(d + 1)(p - 1)}{2}}$ bounds obtained from the results in \cite{strauss1981nonlinear} and \cite{strauss1968decay} become weaker as $\lambda_{1} \searrow 0$ and $\lambda_{2} \nearrow \infty$. This is contrary to our intuition, since one would imagine that, using the decoupling arguments in Section $7$, we should be able to approximate a solution to $(\ref{1.1})$ with initial data given by $(\ref{1.6})$ by a sum of the solutions to $(\ref{1.1})$ with initial data given by each term in $(\ref{1.6})$ separately. This was one motivation for proving Theorem $\ref{t1.1}$.\medskip

If $c_{3} \ll 1$, one could approximate the initial data in $(\ref{1.6})$ by initial data with $c_{3} = 0$, which would then have finite energy. This is exactly the Fourier truncation method. Using the Fourier truncation method, \cite{kenig2000global} proved global well-posedness for $(\ref{1.1})$ when $p = 3$, $d = 3$ (for convenience this will henceforth be called the cubic problem) for initial data in $H^{s} \times H^{s - 1}$ for any $s > \frac{3}{4}$. This method, introduced by \cite{bourgain1998refinements} for the nonlinear Schr{\"o}dinger equation, utilizes the smoothing effect of the Duhamel term
\begin{equation}\label{1.10}
\int_{0}^{t} \frac{\sin((t - \tau) \sqrt{-\Delta})}{\sqrt{-\Delta}} u(\tau)^{3} d\tau.
\end{equation}
The data was then split into a high frequency piece that was small and a low frequency piece that is in $\dot{H}^{1} \times L^{2}$. Global well-posedness holds for $(\ref{1.1})$ with either piece as the initial data (from Theorem $\ref{t1.2}$ and $(\ref{1.7})$). For $s > \frac{3}{4}$, it is possible to ``paste" the two solutions together and obtain a solution to $(\ref{1.1})$ for initial data in $H^{s} \times H^{s - 1}$.

This work was subsequently extended by many authors. See \cite{gallagher2003global}, \cite{bahouri2006global}, \cite{roy2007adapted}, \cite{roy2007global}, \cite{dodson2018global1}, \cite{dodson2019global} for subsequent improvements on this result.\medskip

A second approach that has proven to be very fruitful is the study of type two blowup. There are two different ways in which scattering can fail. The first way is if the $\dot{H}^{s_{c}}$ norm is unbounded. Since the solution to the linear wave equation is a unitary operator, it is clear that one of $(\ref{1.3.1})$--$(\ref{1.3.4})$ would fail. This is called ``type one blowup".

Type two blowup is if scattering fails and yet the $\dot{H}^{s_{c}} \times \dot{H}^{s_{c} - 1}$ norm remain bounded. This behavior is frequently observed for a soliton, although see also the pseudoconformal transformation of the soliton for the nonlinear Schr{\"o}dinger equation (see for example \cite{merle1993determination}).

Type two blowup results have been proved in three dimensions for the intercritical case, that is $\frac{1}{2} < s_{c} < 1$. These results have been proved for both radial \cite{shen2014energy} and nonradial \cite{dodson2020scattering} initial data. The proof uses the concentration compactness argument to exclude the existence of a non-scattering solution of minimal size. It is worth noting that the results in \cite{shen2014energy} hold for both the defocusing equations $(\ref{1.1})$ and the focusing equations:
\begin{equation}\label{1.11}
u_{tt} - \Delta u - |u|^{p - 1} u = 0.
\end{equation}
This is because, unlike in the energy-critical case, there does not exist a soliton solution to $(\ref{1.1})$ that lies in $\dot{H}^{s_{c}} \times \dot{H}^{s_{c} - 1}$. This is in contrast to the focusing, energy--critical problem \cite{kenig2008scattering}. The same should hold in higher dimensions as well.\medskip

Still, for the focusing wave equation, $(\ref{1.11})$, there do exist solutions for which the $\dot{H}^{s_{c}} \times \dot{H}^{s_{c} - 1}$ is unbounded. Indeed, for this equation, the energy is given by
\begin{equation}\label{1.12}
E(u) = \frac{1}{2} \int |\nabla u(t,x)|^{2} dx + \frac{1}{2} \int |u_{t}(t, x)|^{2} dx - \frac{1}{p + 1} \int |u(t,x)|^{p + 1} dx,
\end{equation}
which unlike $(\ref{1.7})$ does not prevent any norms of a solution to $(\ref{1.11})$ from getting arbitrarily large. In fact, it is well known that there exist solutions to $(\ref{1.1})$ for which the $\dot{H}^{s_{c}} \times \dot{H}^{s_{c} - 1}$ norm is unbounded. See \cite{duyckaerts2023global} for the state of the art in this direction and a description of prior results.

\subsection{Outline of the argument}
The proof of Theorem $\ref{t1.1}$ is similar to the argument in \cite{dodson2023sharp}. Specifically, inspired by \cite{kenig2000global}, we split the initial data into two pieces,
\begin{equation}\label{1.13}
u_{0} = v_{0} + w_{0}, \qquad u_{1} = v_{1} + w_{1},
\end{equation}
where $(v_{0}, v_{1})$ has finite conformal energy and $(w_{0}, w_{1})$ has small $\dot{H}^{s_{c}} \times \dot{H}^{s_{c} - 1}$ norm.

Following \cite{strauss1981nonlinear} and \cite{lindblad1995existence}, we know that $(\ref{1.1})$ is scattering with initial data $(v_{0}, v_{1})$ and $(w_{0}, w_{1})$. Therefore, it remains to handle the cross-terms. The contribution of the cross terms $F$ is of the form
\begin{equation}\label{1.14}
\aligned
\langle (t + |x|) Lv + (d - 1) v, (t + |x|) F \rangle, \qquad L = \partial_{t} + \frac{x}{|x|} \cdot \nabla, \\
\langle (t - |x|) \underline{L}v + (d - 1) v, (t + |x|) F \rangle, \qquad \underline{L} = \partial_{t} - \frac{x}{|x|} \cdot \nabla,
\endaligned
\end{equation}
where $F$ contains terms of the form $|w|^{p - 1} |v|$. 

We begin with initial data in a critical Besov space. Following \cite{dodson2022global},
\begin{equation}\label{1.15}
(\ref{1.14}) \lesssim \| (t + |x|) Lv + (d - 1) v \|_{L^{2}} \| (t + |x|) |w|^{\frac{p - 1}{2}} \|_{L^{\infty}} \| w \|_{L_{x}^{p + 1}}^{\frac{p - 1}{2}} \| v \|_{L_{x}^{p + 1}}.
\end{equation}
Combining the radially symmetric Sobolev embedding theorem and the dispersive estimate,
\begin{equation}\label{1.16}
\| (t + |x|) |w|^{\frac{p - 1}{2}} \|_{L^{\infty}} \lesssim 1,
\end{equation}
and therefore $(\ref{1.15})$ implies a bound on
\begin{equation}\label{1.16.1}
\| \frac{\mathcal E(t)}{t^{2}} \|_{L_{t}^{\frac{1}{2(1 - s_{c})}}(\mathbb{R})}.
\end{equation}
Since $\frac{\mathcal E(t)}{t^{2}}$ controls the norm $\| v(t) \|_{L_{x}^{p + 1}}^{p + 1}$ norm, which combined with Strichartz estimates implies
\begin{equation}\label{1.16.2}
\| u \|_{L_{t}^{\frac{p + 1}{2(1 - s_{c})}} L_{x}^{p + 1}(\mathbb{R} \times \mathbb{R}^{d})} < \infty.
\end{equation}
When $s_{c} = \frac{1}{2}$ and $p - 1 = \frac{4}{d - 1}$, $p + 1 = \frac{2(d + 1)}{d - 1}$, which proves $(\ref{1.3})$. For $\frac{1}{2} < s_{c} < 1$ with radially symmetric initial data Lemma $\ref{l2.2}$ combined with $(\ref{1.16.2})$ implies $(\ref{1.3})$. For general initial data, the bound $(\ref{1.16.2})$ need not imply $(\ref{1.3})$.
\begin{remark}
See \cite{dodson2021global} for the non-radial case, where we only obtain bounds for $3 < p \leq 4$. The proof of Lemma $\ref{l2.2}$ is heavily reliant on the radial symmetry.
\end{remark}

For initial data in $\dot{H}^{s_{c}} \times \dot{H}^{s_{c} - 1}$, we still have the bound
\begin{equation}\label{1.17}
\| |x| |w|^{\frac{p - 1}{2}} \|_{L^{\infty}} \lesssim 1,
\end{equation}
from the radial Sobolev embedding theorem,
along with the bound
\begin{equation}\label{1.17.1}
\| t |w|^{\frac{p - 1}{2}} \|_{L^{\infty}(|x| \geq \delta |t|)} \lesssim \frac{1}{\delta}.
\end{equation}

On the other hand, for general initial data in $\dot{H}^{s_{c}} \times \dot{H}^{s_{c} - 1}$, there is no reason to think that the dispersive estimate $\| t |w|^{\frac{p - 1}{2}} \|_{L^{\infty}}$ will hold, since the $\dot{H}^{s_{c}} \times \dot{H}^{s_{c} - 1}$ norm is invariant under the operator
\begin{equation}\label{1.18}
\begin{pmatrix}
\cos(t \sqrt{-\Delta}) & \frac{\sin(t \sqrt{-\Delta})}{\sqrt{-\Delta}} \\ -\sqrt{-\Delta} \sin(t \sqrt{-\Delta}) & \cos(t \sqrt{-\Delta})
\end{pmatrix}.
\end{equation}
Instead, we use that the square $L^{2}$ norm of $\nabla_{t,x} v$ is bounded by the conformal energy divided by $\frac{1}{t^{2}}$, which gives us good decay to cancel out the contribution of $t |w|^{p - 1} |v|$. We use the Morawetz estimate and the local energy decay to do this, which gives a bound on the scattering size.\medskip

In section two, we recall some Strichartz estimates and the small data result of \cite{lindblad1995existence}. In section three we prove some conformal energy estimates and the Besov space result for $\frac{1}{2} < s_{c} < 1$. In section four we prove a Morawetz etimate. In section five, we prove scattering in the $d = 4, 5$ case. 

In section six, we prove a modified small data result in dimensions $d > 5$. This is actually the main new difficulty of the paper. For $s_{c} = \frac{1}{2}$, it is enough to use Strichartz estimates to prove small data scattering results. Such estimates comply well with spatial truncations. When $\frac{1}{2} < s_{c} < 1$, the standard small data arguments utilize fractional derivatives, which are nonlocal operators. Thus, such estimates do not comply easily with spatial truncations. Instead, we use local energy estimates and the radial symmetry of the solution to prove scattering for the small data wave equation.

In section seven, we prove scattering in the $d > 5$ case. Finally, in section eight we complete the proof of Theorem $\ref{t1.1}$ using the profile decomposition.

\section{Strichartz estimates and small data results}
Global well-posedness and scattering for $(\ref{1.1})$ with small initial data is a direct result of Strichartz estimates.
\begin{theorem}[Strichartz estimates]\label{t4.1}
Let $I$ be a time interval and let $u : I \times \mathbb{R}^{d} \rightarrow \mathbb{C}$ be a Schwartz solution to the wave equation,
\begin{equation}\label{4.1}
u_{tt} - \Delta u + F = 0, \qquad u(t_{0}, \cdot) = u_{0}, \qquad u_{t}(t_{0}, \cdot) = u_{1}, \qquad \text{for some} \qquad t_{0} \in I.
\end{equation}
Then for $s \geq 0$, $2 \leq p, \tilde{p} \leq \infty$, $2 \leq q, \tilde{q} < \infty$ obeying the scaling conditions
\begin{equation}\label{4.2}
\frac{1}{p} + \frac{d}{q} = \frac{d}{2} - s = \frac{1}{\tilde{p}'} + \frac{d}{\tilde{q}'} - 2,
\end{equation}
and the wave admissibility conditions
\begin{equation}\label{4.3}
\frac{1}{p} + \frac{d - 1}{2q}, \frac{1}{\tilde{p}} + \frac{d - 1}{2 \tilde{r}} \qquad \leq \qquad \frac{d - 1}{4},
\end{equation}
\begin{equation}\label{4.4}
\| u \|_{L_{t}^{p} L_{x}^{q}(I \times \mathbb{R}^{d})} + \| u \|_{L_{t}^{\infty} \dot{H}_{x}^{s}(I \times \mathbb{R}^{d})} + \| \partial_{t} u \|_{L_{t}^{\infty} \dot{H}_{x}^{s - 1}(I \times \mathbb{R}^{d})} \lesssim_{s, p, q, \tilde{p}, \tilde{q}, s, d} \| u_{0} \|_{\dot{H}_{x}^{s}(\mathbb{R}^{d})} + \| u_{1} \|_{\dot{H}_{x}^{s - 1}(\mathbb{R}^{d})} + \| F \|_{L_{t}^{\tilde{p}'} L_{x}^{\tilde{q}'}(I \times \mathbb{R}^{d})}.
\end{equation}
\end{theorem}
\begin{proof}
This theorem was copied from \cite{tao2006nonlinear}. See \cite{kato1994lq}, \cite{ginibre1995generalized}, \cite{kapitanski1989some}, \cite{lindblad1995existence}, \cite{sogge1995lectures}, \cite{shatah1993regularity}, and \cite{keel1998endpoint} for references.
\end{proof}

To prove small data scattering, by Theorem $\ref{t4.1}$,
\begin{equation}\label{4.5}
\| |\nabla|^{s_{c} - \frac{1}{2}} u \|_{L_{t,x}^{\frac{2(d + 1)}{d - 1}}(\mathbb{R} \times \mathbb{R}^{d})} \lesssim_{d} \| u(0) \|_{\dot{H}_{x}^{s_{c}}(\mathbb{R}^{d})} + \| u_{t}(0) \|_{\dot{H}_{x}^{s_{c} - 1}(\mathbb{R}^{d})} + \| |\nabla|^{s_{c} - \frac{1}{2}} F \|_{L_{t, x}^{\frac{2(d + 1)}{d + 3}}(\mathbb{R} \times \mathbb{R}^{d})},
\end{equation}
which was proved in the original paper \cite{strichartz1977restrictions}. A straightforward application of $(\ref{4.5})$ gives global well-posedness and scattering for $(\ref{1.1})$ with small initial data, for both radially symmetric initial data and general initial data, see \cite{lindblad1995existence}.

\begin{theorem}\label{t4.2}
For any $d > 3$, there exists some $\epsilon_{0}(d) > 0$ such that if $\frac{1}{2} < s_{c} < 1$,
\begin{equation}\label{4.6}
\| u(0, \cdot) \|_{\dot{H}^{s_{c}}(\mathbb{R}^{d})} + \| u_{t}(0, \cdot) \|_{\dot{H}^{s_{c} - 1}(\mathbb{R}^{d})} \leq \epsilon_{0}(d),
\end{equation}
then $(\ref{4.1})$ is globally well-posed and the solution satisfies
\begin{equation}\label{4.7}
\| u \|_{L_{t,x}^{\frac{(d + 1)(p - 1)}{2}}(\mathbb{R} \times \mathbb{R}^{d})} \lesssim \| u(0, \cdot) \|_{\dot{H}^{s_{c}}(\mathbb{R}^{d})} + \| u_{t}(0, \cdot) \|_{\dot{H}^{s_{c} - 1}(\mathbb{R}^{d})}.
\end{equation}
\end{theorem} 
\begin{remark}
This theorem is proved in many places and in far more generality, see for example \cite{tao2006nonlinear}. Observe that, by the fractional product rule and chain rule (see \cite{taylor2000tools}),
\begin{equation}\label{4.7.1}
\| |\nabla|^{s_{c} - \frac{1}{2}} |u|^{p - 1} u \|_{L_{t,x}^{\frac{2(d + 1)}{d + 3}}} \lesssim \| u \|_{L_{t,x}^{\frac{(d + 1)(p - 1)}{2}}}^{p - 1} \| |\nabla|^{s_{c} - \frac{1}{2}} u \|_{L_{t,x}^{\frac{2(d + 1)}{d - 1}}}.
\end{equation}
\end{remark}

While global well-posedness and scattering hold for small nonradial data, the proof in this paper of global well-posedness and scattering for $(\ref{1.1})$ with large initial data relies heavily on radial symmetry. In particular, the proof relies heavily on the radial Strichartz estimate of \cite{sterbenz2005angular}. See also \cite{rodnianskisterbenz}.
\begin{theorem}[Strichartz estimates for radially symmetric initial data]\label{t4.3}
Let $u$ be a radially symmetric function on $\mathbb{R}^{d + 1}$ such that $u_{tt} - \Delta u = 0$. Then, the following estimates hold,
\begin{equation}\label{4.13}
\| u \|_{L_{t}^{p} L_{x}^{q}(\mathbb{R} \times \mathbb{R}^{d})} \lesssim \| u(0) \|_{\dot{H}^{\gamma}} + \| u_{t}(0) \|_{\dot{H}^{\gamma - 1}},
\end{equation}
where 
\begin{equation}\label{4.14}
\frac{1}{p} + \frac{d - 1}{q} < \frac{d - 1}{2}, \qquad \text{and} \qquad \frac{1}{p} + \frac{d}{q} = \frac{d}{2} - \gamma.
\end{equation}
\end{theorem}

The Christ--Kiselev lemma implies that Theorem $\ref{t4.3}$ and Lemma $\ref{l4.4}$ also hold for a small data solution to $(\ref{1.1})$.

\begin{lemma}[Christ--Kiselev lemma]\label{l4.5}
Let $X, Y$ be Banach spaces, let $I$ be a time interval, and let $K \in C^{0}(I \times I \rightarrow B(X \rightarrow Y))$ be a kernel taking values in the space of bounded operators from $X$ to $Y$. If $1 \leq p < q \leq \infty$ is such that
\begin{equation}\label{4.20}
\| \int_{I} K(t, s) f(s) ds \|_{L_{t}^{q}(I \rightarrow Y)} \leq A \| f \|_{L_{t}^{p}(I \rightarrow X)},
\end{equation}
for all $f \in L_{t}^{p}(I \rightarrow X)$ and some $A > 0$, then we also have
\begin{equation}\label{4.21}
\| \int_{s \in I : s < t} K(t, s) f(s) ds \|_{L_{t}^{q}(I \rightarrow Y)} \lesssim_{p, q} A \| f \|_{L_{t}^{p}(I \rightarrow X)}.
\end{equation}
\end{lemma}
\begin{proof}
This lemma was copied out of \cite{tao2006nonlinear}. This lemma was proved in \cite{christ2001maximal}. See also \cite{smith2000global} or \cite{tao2000spherically}.

\end{proof}

Now by Corollary $3.3$ from \cite{dodson2018global}, if $w$ is a solution to the linear wave equation $w_{tt} - \Delta w = 0$, $w(0, x) = P_{j} w_{0}$, $w_{t}(0, x) = P_{j} w_{1}$, for any $R > 0$,
\begin{equation}\label{4.15}
\| |x|^{\frac{d - 2}{2}} w \|_{L_{t}^{2} L_{x}^{\infty}(R \leq |x| \leq 2R)} \lesssim \| P_{j} w_{0} \|_{\dot{H}^{1/2}} + \| P_{j} w_{1} \|_{\dot{H}^{-1/2}}.
\end{equation}
We can use $(\ref{4.15})$ to obtain an estimate for solutions to the linear wave equation.
\begin{lemma}\label{l4.4}
If $w$ solves the linear wave equation $w_{tt} - \Delta w = 0$, then for any $\frac{1}{2} < s_{c} < 1$ and $d > 3$,
\begin{equation}\label{4.17}
\| |x|^{\frac{2}{p - 1} - \frac{1}{2}} w \|_{L_{t}^{2} L_{x}^{\infty}(\mathbb{R} \times \mathbb{R}^{d})} \lesssim_{d, s_{c}} \| u(0) \|_{\dot{H}^{s_{c}}} + \| u_{t}(0) \|_{\dot{H}^{s_{c} - 1}},
\end{equation}
where
\begin{equation}\label{4.16}
s_{c} = \frac{d}{2} - \frac{2}{p - 1}.
\end{equation}
\end{lemma}
\begin{proof}
By the radially symmetric Strichartz estimates,
\begin{equation}\label{4.18}
\| P_{j} w \|_{L_{t}^{2} L_{x}^{\infty}} \lesssim 2^{j(\frac{d - 3}{2})} (\| P_{j} u_{0} \|_{\dot{H}^{1}} + \| P_{j} u_{1} \|_{L^{2}}) \lesssim 2^{j(\frac{2}{p - 1} - \frac{1}{2})} (\| P_{j} u_{0} \|_{\dot{H}^{s_{c}}} + \| P_{j} u_{1} \|_{\dot{H}^{s_{c} - 1}}).
\end{equation}
Therefore,
\begin{equation}\label{4.19}
\| |x|^{\frac{2}{p - 1} - \frac{1}{2}} P_{j} w \|_{L_{t}^{2} L_{x}^{\infty}(|x| \leq 2^{-j})} \lesssim 2^{j(\frac{2}{p - 1} - \frac{1}{2})} (\| P_{j} u_{0} \|_{\dot{H}^{s_{c}}} + \| P_{j} u_{1} \|_{\dot{H}^{s_{c} - 1}}).
\end{equation}
Meanwhile, by Bernstein's inequality and $(\ref{4.17})$,
\begin{equation}\label{4.19.1}
\| |x|^{\frac{d - 2}{2}} P_{j} w \|_{L_{t}^{2} L_{x}^{\infty}(R \leq |x| \leq 2R)} \lesssim 2^{j(\frac{1}{2} - s_{c})} (\| P_{j} u_{0} \|_{\dot{H}^{s_{c}}} + \| P_{j} u_{1} \|_{\dot{H}^{s_{c} - 1}}).
\end{equation}
Summing up $(\ref{4.19.1})$ in $R$,
\begin{equation}\label{4.19.2}
\| |x|^{\frac{2}{p - 1} - \frac{1}{2}} P_{j} w \|_{L_{t}^{2} L_{x}^{\infty}(R \leq |x|)} \lesssim R^{-(s_{c} - \frac{1}{2})} 2^{-j(s_{c} - \frac{1}{2})} (\| P_{j} u_{0} \|_{\dot{H}^{s_{c}}} + \| P_{j} u_{1} \|_{\dot{H}^{s_{c} - 1}}).
\end{equation}
Taking $R = 2^{-j}$, $(\ref{4.19})$ and $(\ref{4.19.2})$ imply that for any $j$,
\begin{equation}\label{4.19.3}
\| |x|^{\frac{2}{p - 1} - \frac{1}{2}} P_{j} w \|_{L_{t}^{2} L_{x}^{\infty}} \lesssim (\| P_{j} u_{0} \|_{\dot{H}^{s_{c}}} + \| P_{j} u_{1} \|_{\dot{H}^{s_{c} - 1}}).
\end{equation}
Furthermore, observe that $(\ref{4.19.3})$ is maximized at $|x| \sim 2^{-j}$. By $(\ref{4.18})$ and $(\ref{4.19.2})$, there exists some $\sigma(p, d) > 0$ so that
\begin{equation}\label{4.19.4}
\| |x|^{\frac{2}{p - 1} - \frac{1}{2}} P_{j} w \|_{L_{t}^{2} L_{x}^{\infty}(|x| \sim 2^{k})} \lesssim 2^{-\sigma |j + k|} (\| P_{j} u_{0} \|_{\dot{H}^{s_{c}}} + \| P_{j} u_{1} \|_{\dot{H}^{s_{c} - 1}}).
\end{equation}
Summing up $(\ref{4.19.4})$ using Young's inequality,
\begin{equation}\label{4.19.5}
\| |x|^{\frac{2}{p - 1} - \frac{1}{2}} w \|_{L_{t}^{2} L_{x}^{\infty}} \lesssim \| u_{0} \|_{\dot{H}^{s_{c}}} + \| u_{1} \|_{\dot{H}^{s_{c} - 1}}.
\end{equation}
\end{proof}

Next, using the local energy estimate and Morawetz estimate in the linear case (see Propositions $\ref{p5.3}$ and $\ref{p5.4}$ for the nonlinear version),
\begin{proposition}\label{p4.6}
If $w$ solves $w_{tt} - \Delta w = 0$, for any $R > 0$,
\begin{equation}\label{4.19.6}
R^{-1} \int \int_{|x| \leq R} |w|^{2} dx dt \lesssim \| w_{0} \|_{L^{2}}^{2} + \| w_{1} \|_{\dot{H}^{-1}}^{2}.
\end{equation}
\end{proposition}
\begin{proof}
Fix $R > 0$ and suppose $2^{-j_{0}} \sim R$ for some $j_{0} \in \mathbb{Z}$. By Theorem $\ref{t4.3}$ and Bernstein's inequality,
\begin{equation}\label{4.19.7}
\| P_{\leq j_{0}} w \|_{L_{t}^{2} L_{x}^{\infty}}^{2} \lesssim 2^{j_{0}(d - 1)} (\| P_{\leq j_{0}} w_{0} \|_{L^{2}}^{2} + \| P_{\leq j_{0}} w_{1} \|_{\dot{H}^{-1}}^{2}).
\end{equation}
Therefore, by H{\"o}lder's inequality, since $R \sim 2^{-j_{0}}$,
\begin{equation}\label{4.19.8}
\| P_{\leq j_{0}} w \|_{L_{t,x}^{2}(|x| \leq R)}^{2} \lesssim R (\| P_{\leq j_{0}} w_{0} \|_{L^{2}}^{2} + \| P_{\leq j_{0}} w_{1} \|_{\dot{H}^{-1}}^{2}).
\end{equation}
Also, if $\chi \in C_{0}^{\infty}(\mathbb{R}^{d})$, $\chi(x) = 1$ for $|x| \leq 1$, $\chi$ is supported on $|x| \leq 2$, if $j > j_{0}$,
\begin{equation}\label{4.19.9}
\aligned
\int \int |\nabla (\chi(\frac{x}{R}) P_{j} w)|^{2} dx dt \lesssim \frac{1}{R^{2}} \int \int_{|x| \leq 2R} |P_{j} w|^{2} dx dt + \int \int_{|x| \leq 2R} |\nabla P_{j} w|^{2} dx dt
\lesssim \| P_{j} w_{0} \|_{\dot{H}^{1}}^{2} + \| P_{j} w_{1} \|_{L^{2}}^{2}.
\endaligned
\end{equation}
Therefore, by Bernstein's inequality, for any $k$,
\begin{equation}\label{4.19.10}
\int \int |P_{k} (\chi(\frac{x}{R}) P_{j} w)|^{2} dx dt \lesssim 2^{j - k} R (\| P_{j} w_{0} \|_{L^{2}}^{2} + \| P_{j} w_{1} \|_{\dot{H}^{-1}}^{2}).
\end{equation}
Next, by the Fourier support properties of $P_{j} w$ and the radial Strichartz estimate, for any $N$,
\begin{equation}\label{4.19.11}
\aligned
\int \int |P_{\leq j - 3}(\chi(\frac{x}{R}) (P_{j} w))|^{2} dx dt \lesssim R^{2} \| P_{\leq j - 3}(\chi(\frac{x}{R}) (P_{j} w)) \|_{L_{t}^{2} L_{x}^{\frac{2d}{d - 2}}}^{2} \\ \lesssim R^{2} 2^{j} R^{-N} 2^{-jN} (\| P_{j} w_{0} \|_{L^{2}}^{2} + \| P_{j} w_{1} \|_{L^{2}}^{2}) \lesssim R 2^{-jN} 2^{j_{0}N} (\| P_{j} w_{0} \|_{L^{2}}^{2} + \| P_{j} w_{1} \|_{L^{2}}^{2}).
\endaligned
\end{equation}
Combining $(\ref{4.19.8})$, $(\ref{4.19.10})$ and $(\ref{4.19.11})$ with Plancherel's theorem,
\begin{equation}\label{4.19.12}
\int \int |\chi(\frac{x}{R}) w|^{2} dx dt \lesssim R (\| w_{0} \|_{L^{2}}^{2} + \| w_{1} \|_{\dot{H}^{-1}}^{2}).
\end{equation}
\end{proof}

\begin{corollary}\label{c4.7}
If $w$ solves
\begin{equation}
w_{tt} - \Delta w = F, \qquad w(0, x) = 0, \qquad w_{t}(0, x) = 0,
\end{equation}
\begin{equation}
\| w \|_{L_{t}^{\infty} \dot{H}^{1}} + \| w_{t} \|_{L_{t}^{\infty} L_{x}^{2}} \lesssim \sum_{k} 2^{k/2} \| F \|_{L_{t,x}^{2}(2^{k} \leq |x| \leq 2^{k + 1})}.
\end{equation}
\end{corollary}
\begin{proof}
This follows from Proposition $\ref{p4.6}$ and standard duality arguments for Strichartz estimates.
\end{proof}

\section{Conformal energy and Morawetz estimates}
For large initial data, global well-posedness and scattering for $(\ref{1.1})$ is equivalent to proving that $(\ref{1.1})$ has a solution which satisfies
\begin{equation}\label{5.1}
\| u \|_{L_{t,x}^{\frac{(d + 1)(p - 1)}{2}}(\mathbb{R} \times \mathbb{R}^{d})} < \infty.
\end{equation}
Indeed, if $(\ref{5.1})$ holds, $\mathbb{R}$ can be partitioned into finitely many subintervals $I_{j}$ for which
\begin{equation}\label{5.2}
\| u \|_{L_{t,x}^{\frac{(d + 1)(p - 1)}{2}}(I_{j} \times \mathbb{R}^{d})} \ll 1.
\end{equation}
One can then use the Picard iteration argument from Theorem $\ref{t4.2}$ to prove global well-posedness and scattering.\medskip

On the other hand, if scattering is known to occur, then by $(\ref{4.5})$ and the Picard iteration argument from Theorem $\ref{t4.2}$, $(\ref{5.1})$ holds.\medskip

For radially symmetric initial data, to prove $(\ref{5.1})$, it is enough to prove that
\begin{equation}\label{5.2.1}
\| u \|_{L_{t}^{\frac{p + 1}{2(1 - s_{c})}} L_{x}^{p + 1}(\mathbb{R} \times \mathbb{R}^{d})} < \infty.
\end{equation}
\begin{lemma}\label{l2.2}
Let $u$ be a global, radially symmetric solution to $(\ref{1.1})$ that satisfies $(\ref{5.2.1})$. Then,
\begin{equation}\label{2.11}
\| u \|_{L_{t, x}^{\frac{(d + 1)(p - 1)}{2}}(\mathbb{R} \times \mathbb{R}^{d})} < \infty.
\end{equation}
\end{lemma}
\begin{proof}
Partition $\mathbb{R}$ into finitely many intervals $I_{j}$ for which
\begin{equation}\label{2.12}
\| u \|_{L_{t}^{\frac{p + 1}{2(1 - s_{c})}} L_{x}^{p + 1}(I_{j} \times \mathbb{R}^{d})} \leq \epsilon,
\end{equation}
on each subinterval, for some $0 < \epsilon \ll 1$. Suppose $I_{j} = [a_{j}, b_{j}]$. By Duhamel's principle, for $t \in I_{j}$,
\begin{equation}\label{2.13}
u(t) = \cos((t - a_{j}) \sqrt{-\Delta}) u(a_{j}) + \frac{\sin((t - a_{j}) \sqrt{-\Delta})}{\sqrt{-\Delta}} u_{t}(a_{j}) - \int_{a_{j}}^{t} \frac{\sin((t - \tau) \sqrt{-\Delta})}{\sqrt{-\Delta}} |u(\tau)|^{p - 1} u(\tau) d\tau.
\end{equation}
Using the radially symmetric Strichartz estimates in Theorem $\ref{t4.3}$ and Lemma $\ref{l4.5}$ along with the fractional product rule in \cite{christ1991dispersion} (see also \cite{taylor2000tools}), we have
\begin{equation}\label{2.14}
\aligned
\| |\nabla|^{s_{c} - \frac{1}{2}} u \|_{L_{t}^{2} L_{x}^{\frac{2d}{d - 2}}(I_{j} \times \mathbb{R}^{d})} \lesssim \| u(a_{j}) \|_{\dot{H}^{s_{c}}} + \| u_{t}(a_{j}) \|_{\dot{H}^{s_{c} - 1}} + \| |\nabla|^{s_{c} - 1/2} u \|_{L_{t}^{2} L_{x}^{\frac{2d}{d - 2}}(I_{j} \times \mathbb{R}^{d})} \| u \|_{L_{t}^{\frac{p + 1}{2(1 - s_{c})}} L_{x}^{p + 1}(I_{j} \times \mathbb{R}^{d})}^{p - 1} \\
\lesssim \| u(a_{j}) \|_{\dot{H}^{s_{c}}} + \| u_{t}(a_{j}) \|_{\dot{H}^{s_{c} - 1}} + \epsilon^{p - 1} \| |\nabla|^{s_{c} - 1/2} u \|_{L_{t}^{2} L_{x}^{\frac{2d}{d - 2}}(I_{j} \times \mathbb{R}^{d})}.
\endaligned
\end{equation}
For $0 < \epsilon \ll 1$ sufficiently small,
\begin{equation}\label{2.14.1}
\| |\nabla|^{s_{c} - \frac{1}{2}} u \|_{L_{t}^{2} L_{x}^{\frac{2d}{d - 2}}(I_{j} \times \mathbb{R}^{d})} \lesssim \| u(a_{j}) \|_{\dot{H}^{s_{c}}} + \| u_{t}(a_{j}) \|_{\dot{H}^{s_{c} - 1}}.
\end{equation}
Notice that $0 < \epsilon \ll 1$ does not depend at all on $\| u(a_{j}) \|_{\dot{H}^{s_{c}}}$ and $\| u_{t}(a_{j}) \|_{\dot{H}^{s_{c} - 1}}$. Furthermore, combining Theorems $\ref{t4.1}$ and $\ref{t4.3}$, Lemma $\ref{l4.5}$, $(\ref{2.14})$, and $(\ref{2.14.1})$,
\begin{equation}\label{2.14.2}
\| u \|_{L_{t,x}^{\frac{(d + 1)(p - 1)}{2}} \cap L_{t}^{\infty} \dot{H}^{s_{c}}(I_{j} \times \mathbb{R}^{d})} + \| u_{t} \|_{L_{t}^{\infty} \dot{H}^{s_{c} - 1}(I_{j} \times \mathbb{R}^{d})} \lesssim \| u(a_{j}) \|_{\dot{H}^{s_{c}}} + \| u_{t}(a_{j}) \|_{\dot{H}^{s_{c} - 1}}.
\end{equation}
Iterating $(\ref{2.14.2})$ finitely many times proves $(\ref{2.11})$.
\end{proof}

For large data, \cite{strauss1981nonlinear} and \cite{strauss1968decay} proved global well-posedness and scattering for $(\ref{1.1})$ with large initial data with sufficient regularity and decay. The proof uses the conformal energy.
\begin{theorem}\label{t5.1}
Suppose $u_{0}$ and $u_{1}$ are initial data that satisfy
\begin{equation}\label{5.3}
\| \langle x \rangle \nabla u_{0} \|_{L^{2}} + \| u_{0} \|_{L^{2}} + \| \langle x \rangle u_{1} \|_{L^{2}} + \| \langle x \rangle^{\frac{2}{p + 1}} u_{0} \|_{L_{x}^{p + 1}} < \infty.
\end{equation}
Here, $\langle x \rangle = (1 + |x|^{2})^{1/2}$. Then the solution to $(\ref{1.1})$ is globally well-posed and scattering.
\end{theorem}
\begin{proof}
The conformal energy,
\begin{equation}\label{5.4}
\aligned
\mathcal E(u) = \frac{1}{4} \int_{\mathbb{R}^{d}} |(t + |x|) Lu + (d - 1) u|^{2} + |(t - |x|) \underline{L} u + (d - 1) u|^{2} dx \\ + \frac{1}{2} \int (t^{2} + |x|^{2}) |\cancel{\nabla} u|^{2} dx + \frac{1}{p + 1} \int (t^{2} + |x|^{2}) |u|^{p + 1} dx,
\endaligned
\end{equation}
is a decreasing quantity, where $L = (\partial_{t} + \frac{x}{|x|} \cdot \nabla)$ and $\underline{L} = (\partial_{t} - \frac{x}{|x|} \cdot \nabla)$.\medskip

Now then, translating in time so that the initial data is at time $t = 1$, $(\ref{5.3})$ implies that $\mathcal E(1) < \infty$. Since $\mathcal E(t) \leq \mathcal E(1)$,
\begin{equation}\label{5.15}
\| \frac{\mathcal E(t)}{t^{2}} \|_{L_{t}^{\frac{1}{2(1 - s_{c})}}([1, \infty)} < \infty.
\end{equation}
Since $\| u \|_{L^{p + 1}}^{p + 1} \lesssim \frac{\mathcal E(t)}{t^{2}}$, $(\ref{5.15})$ implies $(\ref{5.2.1})$.
\end{proof}

As a warm-up to proving Theorem $\ref{t1.1}$, we prove scattering for $(\ref{1.1})$ with radially symmetric initial data satisfying
\begin{equation}\label{5.15.1}
\| u_{0} \|_{B_{1,1}^{\frac{d}{2} + s_{c}}} + \| u_{1} \|_{B_{1,1}^{\frac{d}{2} + s_{c} - 1}} = A < \infty.
\end{equation}
\begin{theorem}
Theorem $\ref{t1.1}$ holds for radial initial data in $B_{1,1}^{\frac{d}{2} + s_{c}} \times B_{1,1}^{\frac{d}{2} + s_{c} - 1}$. Furthermore, the scattering norm $(\ref{5.2.1})$ depends polynomially on $A$ in $(\ref{5.15.1})$.
\end{theorem}
\begin{proof}
The proof follows \cite{dodson2022global}. Let $u$ be a solution to $(\ref{1.1})$ with initial data that satisfies $(\ref{5.15.1})$. Split the solution $u = w + v$, where $v$ and $w$ solve
\begin{equation}\label{5.15.2}
\aligned
w_{tt} - \Delta w = 0, \qquad w(0, x) &= u_{0}, \qquad w_{t}(0,x) = u_{1}, \\
v_{tt} - \Delta v + |u|^{p - 1} u = 0, \qquad v(0, x) &= 0, \qquad v_{t}(0, x) = 0.
\endaligned
\end{equation}
Now define the conformal energy
\begin{equation}\label{5.15.3}
\aligned
\mathcal E(t) = \frac{1}{4} \int_{\mathbb{R}^{d}} |(t + |x|) Lv + (d - 1) v|^{2} + |(t - |x|) \underline{L} v + (d - 1) v|^{2} dx \\ + \frac{1}{2} \int (t^{2} + |x|^{2}) |\cancel{\nabla} v|^{2} dx + \frac{1}{p + 1} \int (t^{2} + |x|^{2}) |v|^{p + 1} dx.
\endaligned
\end{equation}
Then by direct computation,
\begin{equation}\label{5.15.4}
\aligned
\frac{d}{dt} \mathcal E(t) \leq -\frac{1}{2} \langle (t + |x|) Lv + (d - 1)v, (t + |x|) \{ |u|^{p - 1} u - |v|^{p - 1} v \} \rangle \\ - \frac{1}{2} \langle (t - |x|) \underline{L}v + (d - 1)v, (t - |x|) \{ |u|^{p - 1} u - |v|^{p - 1} v \} \rangle \\
\lesssim \mathcal E(t)^{1/2} \| (t \pm |x|) |w|^{\frac{p - 1}{2}} \|_{L^{\infty}} \| w \|_{L^{p + 1}}^{\frac{3 - p}{2}} \| v \|_{L^{p + 1}}^{p - 1} + \mathcal E(t)^{1/2} \| (t \pm |x|) |w|^{\frac{p - 1}{2}} \|_{L^{\infty}} \| w \|_{L^{p + 1}}^{\frac{p + 1}{2}}.
\endaligned
\end{equation}
Combining the dispersive estimate with the radial Sobolev embedding theorem,
\begin{equation}\label{5.15.5}
\| (t + |x|) |w|^{\frac{p - 1}{2}} \|_{L^{\infty}} + \| (t - |x|) |w|^{\frac{p - 1}{2}} \|_{L^{\infty}} \lesssim A^{\frac{p - 1}{2}}.
\end{equation}
Therefore, since $\mathcal E(0) = 0$, by the fundamental theorem of calculus,
\begin{equation}\label{5.15.6}
\mathcal E(t) \lesssim A^{\frac{p - 1}{2}} \int_{0}^{t} \frac{\mathcal E(\tau)^{\frac{1}{2} + \frac{p - 1}{p + 1}}}{\tau^{\frac{2(p - 1)}{p + 1}}} \| w(\tau) \|_{L^{p + 1}}^{\frac{3 - p}{2}} d\tau + A^{\frac{p - 1}{2}} \int_{0}^{t} \mathcal E(\tau)^{1/2} \| w(\tau) \|_{L^{p + 1}}^{\frac{p + 1}{2}} d\tau.
\end{equation}
Now, for any $r \geq 1$, by Fubini's theorem,
\begin{equation}\label{5.15.7}
\aligned
\int_{0}^{\infty} \frac{1}{t^{2r}} (\int_{0}^{t} \mathcal E(\tau)^{\frac{1}{2} + \frac{p - 1}{p + 1}} \frac{1}{\tau^{\frac{2(p - 1)}{p + 1}}} f(\tau)^{\frac{3 - p}{2(p + 1)}} d\tau)^{r} dt \lesssim_{r} \int_{0}^{\infty} \mathcal E(\tau)^{r(\frac{1}{2} + \frac{p - 1}{p + 1})} \frac{1}{\tau^{r(1 + \frac{2(p - 1)}{p + 1})}} f(\tau)^{r(\frac{3 - p}{2(p + 1)})} d\tau \\
\lesssim (\int_{0}^{\infty} \frac{\mathcal E(t)^{r}}{t^{2r}} dt)^{\frac{1}{2} + \frac{p - 1}{p + 1}} (\int_{0}^{\infty} f(t)^{r \frac{3 - p}{2(p + 1)}} dt)^{\frac{3 - p}{2(p + 1)}}.
\endaligned
\end{equation}
Also, by Fubini's theorem,
\begin{equation}\label{5.15.8}
\aligned
\int_{0}^{\infty} \frac{1}{t^{2r}} (\int_{0}^{t} \mathcal E(\tau)^{\frac{1}{2}}  f(\tau)^{\frac{1}{2}} d\tau)^{r} dt \lesssim_{r}  (\int_{0}^{\infty} \frac{\mathcal E(t)^{r}}{t^{2r}} dt)^{\frac{1}{2}} (\int_{0}^{\infty} f(t)^{r} dt)^{\frac{1}{2}}.
\endaligned
\end{equation}
Since $\frac{p - 1}{p + 1} < \frac{1}{2}$ when $p < 3$,
\begin{equation}\label{5.15.9}
\int_{0}^{\infty} (\frac{\mathcal E(t)}{t^{2}})^{\frac{1}{2(1 - s_{c})}} dt \lesssim 1 + A^{\sigma(p, d)},
\end{equation}
for some $\sigma(p, d) < \infty$.
\end{proof}

\section{Morawetz estimate}
Now compute a Morawetz estimate.
\begin{proposition}\label{p5.2}
For any $T > 0$, if $u$ solves $(\ref{1.1})$,
\begin{equation}\label{5.16}
\aligned
\int_{0}^{T} \int [\frac{1}{|x|^{3}} u^{2} + \frac{1}{|x|} |u|^{p + 1}] dx dt \lesssim \sup_{t \in [0, T]} \| \nabla_{t, x} u \|_{L^{2}}^{2}.
\endaligned
\end{equation}
\end{proposition}
\begin{proof}
Let $M(t)$ denote the Morawetz potential for radially symmetric solutions,
\begin{equation}\label{5.17}
M(t) = \int u_{t} u_{r} r^{d - 1} dr + \frac{d - 1}{2} \int u_{t} u r^{d - 2} dr.
\end{equation}
Using Hardy's inequality,
\begin{equation}\label{5.18}
\sup_{0 \leq t \leq T} |M(t)| \lesssim \sup_{t \in [0, T]} \| \nabla u \|_{L^{2}}^{2} + \| u_{t} \|_{L^{2}}^{2}.
\end{equation}
Next, by the product rule,
\begin{equation}\label{5.19}
\frac{d}{dt} M(t) = \int u_{t} u_{rt} r^{d - 1} dr + \frac{d - 1}{2} \int u_{t}^{2} r^{d - 2} dr + \int u_{tt} u_{r} r^{d - 1} dr + \frac{d - 1}{2} \int u_{tt} u r^{d - 2} dr.
\end{equation}
Integrating by parts, 
\begin{equation}\label{5.20}
\int u_{t} u_{rt} r^{d - 1} dr + \frac{d - 1}{2} \int u_{t}^{2} r^{d - 2} dr = 0.
\end{equation}
Next, integrating by parts, since $\Delta u = u_{rr} + \frac{d - 1}{r} u_{r}$,
\begin{equation}\label{5.21}
\aligned
\int \Delta u u_{r} r^{d - 1} dr + \frac{d - 1}{2} \int \Delta u u r^{d - 2} dr = -\frac{(d - 1)(d - 3)}{2} \int  u^{2} r^{d - 4} dr.
\endaligned
\end{equation}
Next, integrating by parts,
\begin{equation}\label{5.22}
\aligned
-\int  |u|^{p - 1} u u_{r} r^{d - 1} - \frac{d - 1}{2} \int |u|^{p + 1} r^{d - 2} dr = -(d - 1)(\frac{1}{2} - \frac{1}{p + 1}) \int |u|^{p + 1} r^{d - 2} dr.
\endaligned
\end{equation}
Since $(\ref{5.21})$ and $(\ref{5.22})$ have the same sign, the proof is complete.
\end{proof}

Finite propagation speed also allows us to cut-off in space, which will be important to the proof of scattering.
\begin{proposition}\label{p5.3}
Suppose $u$ solves $(\ref{1.1})$. For any $T > 0$, if $\chi \in C_{0}^{\infty}(\mathbb{R}^{d})$, $\chi(x) = 1$ for $|x| \leq 1$, $\chi(x) = 0$ for $|x| \geq 2$, then for any $\delta > 0$,
\begin{equation}\label{5.24}
\aligned
\int_{0}^{T} \int \chi(\frac{x}{\delta T}) [\frac{1}{|x|^{3}} u^{2} + \frac{1}{|x|} |u|^{p + 1}] dx dt \lesssim_{\delta} \sup_{t \in [0, T]} \| \nabla_{t, x} u \|_{L^{2}(|x| \leq 2 \delta T)}^{2} + \sup_{t \in [0, T]} \frac{1}{\delta^{2} T^{2}} \| u \|_{L^{2}(|x| \leq 2 \delta T)}^{2}.
\endaligned
\end{equation}
\end{proposition}
\begin{proof}
This time use
\begin{equation}\label{5.25}
M(t) = \int \chi(\frac{x}{\delta T}) u_{t} u_{r} r^{d - 1} dr + \frac{d - 1}{2} \int \chi(\frac{x}{\delta T}) u u_{t} r^{d - 2} dr.
\end{equation}
We can use the same computations in $(\ref{5.17})$--$(\ref{5.22})$, only we also have to take into account the fact that when integrating by parts, derivatives can hit $\chi(\frac{x}{\delta T})$. Now then, since
\begin{equation}\label{5.26}
|\nabla^{(k)} \chi(\frac{x}{\delta T})| \lesssim_{k} \frac{1}{\delta^{k} T^{k}}, \qquad \text{for} \qquad k = 1, 2, 3, \qquad \text{and is supported on} \qquad \delta T \leq |x| \leq 2 \delta T.
\end{equation}
Moreover, for any $l \geq 0$,
\begin{equation}\label{5.27}
\frac{1}{r^{l}} |\nabla^{(k)} \chi(\frac{x}{\delta T})| \lesssim_{k, l} \frac{1}{\delta^{k + l} T^{k + l}}, \qquad \text{for} \qquad k = 1, 2, 3, \qquad \text{and is supported on} \qquad \delta T \leq |x| \leq 2 \delta T.
\end{equation}
Therefore, the contribution of the additional terms coming from $\chi(\frac{x}{\delta T})$ is bounded by
\begin{equation}\label{5.28}
\aligned
\frac{1}{\delta T} \int_{0}^{T} \int_{\delta T \leq |x| \leq 2 \delta T} |\nabla_{t,x} u|^{2} dx dt + \frac{1}{\delta^{3} T^{3}} \int_{0}^{T} \int_{\delta T \leq |x| \leq 2 \delta T} |u|^{2} dx dt  \\ \lesssim_{\delta} \sup_{t \in [0, T]} \| \nabla_{t, x} u \|_{L^{2}(|x| \leq 2 \delta T)}^{2} + \sup_{t \in [0, T]} \frac{1}{\delta^{2} T^{2}} \| u \|_{L^{2}(|x| \leq 2 \delta T)}^{2}.
\endaligned
\end{equation}
\end{proof}

Next, we prove a local energy decay estimate.
\begin{proposition}\label{p5.4}
For any $T > 0$, $R > 0$, if $u$ solves $(\ref{1.1})$,
\begin{equation}\label{5.29}
\aligned
R^{-1} \int_{0}^{T} \int_{|x| \leq R} \chi(\frac{x}{\delta T}) [|\nabla u|^{2} + u_{t}^{2} + |u|^{p + 1}] dx dt \lesssim_{\delta} \sup_{t \in [0, T]} \| \nabla_{t, x} u \|_{L^{2}(|x| \leq 2 \delta T)}^{2} \\ + \sup_{t \in [0, T]} \frac{1}{\delta^{2} T^{2}} \| u \|_{L^{2}(|x| \leq 2 \delta T)}^{2} + \sup_{t \in [0, T]} \| u \|_{L^{\frac{2(d + 1)}{d - 1}}(|x| \leq 2 \delta T)}^{\frac{2(d + 1)}{d - 1}}.
\endaligned
\end{equation}
\end{proposition}
\begin{proof}
Define $\psi(r) \in C_{0}^{\infty}(\mathbb{R}^{d})$ and suppose $\psi(r) = 1$ for $0 \leq r \leq 1$, $\psi(r) = \frac{3}{2r}$ for $r > 2$, and $\partial_{r}(r \psi(r)) = \phi(r) \geq 0$ for $r \geq 0$. Now, define the Morawetz potential
\begin{equation}\label{5.30}
M(t) = R^{-1} \int \chi(\frac{r}{\delta T}) \psi(\frac{r}{R}) u_{t} u_{r} r^{d} dr + \frac{d - 1}{2} R^{-1} \int \chi(\frac{r}{\delta T}) \psi(\frac{r}{R}) u_{t} u r^{d - 1} dr.
\end{equation}
As in $(\ref{5.25})$,
\begin{equation}\label{5.31}
\sup_{0 \leq t \leq T} M(t) \lesssim_{\delta} \sup_{t \in [0, T]} \| \nabla_{t, x} u \|_{L^{2}(|x| \leq 2 \delta T)}^{2} + \sup_{t \in [0, T]} \frac{1}{\delta^{2} T^{2}} \| u \|_{L^{2}(|x| \leq 2 \delta T)}^{2}. 
\end{equation}

Next, by direct computation,
\begin{equation}\label{5.32}
\aligned
\frac{d}{dt} M(t) = R^{-1} \int \chi(\frac{r}{\delta T}) \psi(\frac{r}{R}) u_{t} u_{tr} r^{d} dr + \frac{d - 1}{2} R^{-1} \int \chi(\frac{r}{\delta T}) \psi(\frac{r}{R}) u_{t}^{2} r^{d - 1} dr \\
+ R^{-1} \int \chi(\frac{r}{\delta T}) \psi(\frac{r}{R}) u_{tt} u_{r} r^{d} dr + \frac{d - 1}{2} R^{-1} \int \chi(\frac{r}{\delta T}) \psi(\frac{r}{R}) u_{tt} u r^{d - 1} dr.
\endaligned
\end{equation}
Integrating by parts in $r$, by $(\ref{5.27})$,
\begin{equation}\label{5.33}
\aligned
R^{-1} \int \chi(\frac{r}{\delta T}) \psi(\frac{r}{R}) u_{t} u_{tr} r^{d} dr + \frac{d - 1}{2} R^{-1} \int \chi(\frac{r}{\delta T}) \psi(\frac{r}{R}) u_{t}^{2} r^{d - 1} dr \\ \leq -\frac{1}{2} \int \chi(\frac{r}{\delta T}) \phi(\frac{r}{R}) u_{t}^{2} r^{d - 1} dr + O(\frac{1}{\delta T} \int_{\delta T \leq |x| \leq 2 \delta T} u_{t}^{2} dx).
\endaligned
\end{equation}

Next, integrating by parts and using $(\ref{5.27})$,
\begin{equation}\label{5.34}
\aligned
R^{-1} \int \chi(\frac{r}{\delta T}) \psi(\frac{r}{R}) (u_{rr} + \frac{d - 1}{r} u_{r}) u_{r} r^{d} dr + \frac{d - 1}{2} R^{-1} \chi(\frac{r}{\delta T}) \psi(\frac{r}{R}) (u_{rr} + \frac{d - 1}{r} u_{r}) u r^{d - 1} dr \\
\leq -\frac{1}{2} \int \chi(\frac{r}{\delta T}) \phi(\frac{r}{R}) u_{r}^{2} r^{3} dr + O(\frac{1}{\delta T} \int_{\delta T \leq |x| \leq 2 \delta T} |\nabla u|^{2} dx) \\ + O(\frac{1}{R} \int \chi(\frac{r}{\delta T}) \psi(\frac{r}{R}) u^{2} r^{d - 3} dr) + O(\frac{1}{\delta^{3} T^{3}} \int_{\delta T \leq |x| \leq 2 \delta T} u^{2} dx).
\endaligned
\end{equation}
Now then,
\begin{equation}\label{5.35}
\aligned
\frac{1}{\delta T} \int_{0}^{T} \int_{\delta T \leq |x| \leq 2 \delta T} |\nabla u|^{2} dx dt + \frac{1}{\delta^{3} T^{3}} \int_{0}^{T} \int_{\delta T \leq |x| \leq 2 \delta T} u^{2} dx dt \\ \lesssim_{\delta} \sup_{t \in [0, T]} \| \nabla_{t, x} u \|_{L^{2}(|x| \leq 2 \delta T)}^{2} + \sup_{t \in [0, T]} \frac{1}{\delta^{2} T^{2}} \| u \|_{L^{2}(|x| \leq 2 \delta T)}^{2}.
\endaligned
\end{equation}
Also,
\begin{equation}\label{5.36}
\frac{1}{R} \int_{0}^{T} \int \chi(\frac{r}{\delta T}) \psi(\frac{r}{R}) u^{2} r^{d - 3} dr dt \lesssim \int_{0}^{T} \int \chi(\frac{r}{\delta T}) \frac{1}{|x|^{3}} u^{2} dx dt,
\end{equation}
and we can use Proposition $\ref{p5.3}$ to estimate this term.\medskip

Next, integrating by parts,
\begin{equation}\label{5.37}
\aligned
-R^{-1} \int \chi(\frac{r}{\delta T}) \psi(\frac{r}{R})  |u|^{p - 1} u u_{r} r^{d} dr - \frac{d - 1}{2} R^{-1} \int \chi(\frac{r}{\delta T}) \psi(\frac{r}{R}) |u|^{p + 1} r^{d - 1} dr \\
=  -\frac{1}{R}(\frac{d}{p + 1} - \frac{d - 1}{2}) \int \chi(\frac{r}{\delta T}) \psi(\frac{r}{R}) |u|^{p + 1} r^{d - 1} dr \\ + \frac{1}{R^{2}} \int \chi(\frac{r}{\delta T}) \psi'(\frac{r}{R}) |u|^{p + 1} r^{d} + \frac{1}{\delta T} \int_{\delta T \leq |x| \leq 2 \delta T} |u|^{p + 1} dx \\
\leq  -\frac{d - 1}{2R(d + 1)} \int \chi(\frac{r}{\delta T}) \psi(\frac{r}{R}) |u|^{p + 1} r^{d - 1} dr  + O(\frac{1}{\delta T} \int_{\delta T \leq |x| \leq 2 \delta T} |u|^{p + 1} dx).
\endaligned
\end{equation}
The last inequality uses the fact that $\psi'(r) \leq 0$ for all $r$ and the fact that since $p + 1 > \frac{2(d + 1)}{d - 1}$, $(d - 1)(\frac{d}{p + 1} - \frac{d - 1}{2}) \leq -\frac{d - 1}{2(d + 1)}$. This completes the proof of the theorem.
\end{proof}

Now suppose that $v$ solves
\begin{equation}\label{5.23}
v_{tt} - \Delta v + |v|^{p - 1} v = F.
\end{equation}
\begin{proposition}\label{p5.5}
Suppose $v$ solves $(\ref{5.23})$. Then for any $T > 0$, $R > 0$,
\begin{equation}\label{5.38}
\aligned
R^{-1} \int_{0}^{T} \int_{|x| \leq R} \chi(\frac{x}{\delta T}) [|\nabla v|^{2} + v_{t}^{2} + |v|^{p + 1}] dx dt + \int_{0}^{T} \int \chi(\frac{x}{\delta T}) [\frac{1}{|x|^{3}} |v|^{2} + \frac{1}{|x|} |v|^{p + 1}] dx dt \\\lesssim_{\delta} \sup_{t \in [0, T]} \| \nabla_{t, x} v \|_{L^{2}(|x| \leq 2 \delta T)}^{2} + \sup_{t \in [0, T]} \frac{1}{\delta^{2} T^{2}} \| v \|_{L^{2}(|x| \leq 2 \delta T)}^{2} + \sup_{t \in [0, T]} \| v \|_{L^{p + 1}(|x| \leq 2 \delta T)}^{p + 1} \\
+ \int_{0}^{T} \int \chi(\frac{x}{\delta T}) |v_{r}| |F| dx dt + \int_{0}^{T} \int \chi(\frac{x}{\delta T}) |v| |F| \frac{1}{|x|} dx dt.
\endaligned
\end{equation}
\end{proposition}
\begin{proof}
The proof uses the same computations as in the proof of Propositions $\ref{p5.3}$ and $\ref{p5.4}$ and $(\ref{5.23})$.
\end{proof}

\section{Scattering in the $d = 4$ and $d = 5$ case}
The proofs of global well-posedness and scattering are slightly different in the $d = 4$ and $d = 5$ cases, and the $d > 5$ cases. We start with the $d = 4, 5$ case.
\begin{theorem}\label{t2.1}
If $u$ is a solution to the nonlinear wave equation,
\begin{equation}\label{2.2}
u_{tt} - \Delta u + |u|^{p - 1} u = 0, \qquad u(0,x) = u_{0} \in \dot{H}^{s_{c}}, \qquad u_{t}(0, x) = u_{1} \in \dot{H}^{s_{c} - 1}, \qquad u_{0}, u_{1} \qquad \text{radial},
\end{equation}
$u : \mathbb{R} \times \mathbb{R}^{d} \rightarrow \mathbb{R}$, and $s_{c} = \frac{d}{2} - \frac{2}{p - 1}$, then when $d = 4, 5$, $u$ is a global solution to $(\ref{2.2})$ and scatters, that is
\begin{equation}\label{2.3}
\| u \|_{L_{t}^{\frac{p + 1}{2(1 - s_{c})}} L_{x}^{p + 1}(\mathbb{R} \times \mathbb{R}^{d})} \leq C(u_{0}, u_{1}) < \infty.
\end{equation}
\end{theorem}
\begin{remark}
As in \cite{dodson2023sharp1}, the bound obtained in Theorem $\ref{t2.1}$ does not depend on the $\dot{H}^{s_{c}} \times \dot{H}^{s_{c} - 1}$ norm, but also on the profile of the initial data. As usual, such as in \cite{dodson2023sharp1}, we can obtain a bound on the right hand side that depends on the $\| u_{0} \|_{\dot{H}^{s_{c}}} + \| u_{1} \|_{\dot{H}^{s_{c} - 1}}$ norm only by using the profile decomposition.

Also, using the approximation argument in \cite{dodson2023sharp1}, it is enough to approximate $u_{0}$ and $u_{1}$ with Schwarz functions $u_{0}^{\sigma}$ and $u_{1}^{\sigma}$, where $u_{0}^{\sigma} \rightarrow u_{0}$ in $\dot{H}^{s_{c}}$ and $u_{1}^{\sigma} \rightarrow u_{1}$ in $\dot{H}^{s_{c} - 1}$ as $\sigma \searrow 0$. We know from \cite{strauss1968decay} that $(\ref{2.3})$ holds for any $\sigma > 0$, so it only remains to show that the right hand side is uniform in $\sigma > 0$.

Finally, again following \cite{dodson2023sharp1}, time reversal symmetry implies that it is enough to prove $(\ref{2.3})$ forward in time, on $[0, \infty)$. To avoid complications from small $t$, we further shift the initial data to $t = 1$ and solve forward in time.
\end{remark}

\begin{remark}
As a consequence of the proof, we show that a solution $u$ to $(\ref{1.1})$ has the form $u = v + w$, where $w$ is a small data solution to $(\ref{1.1})$ and for all $t \geq 1$,
\begin{equation}
\| \nabla_{t,x} v \|_{L_{x}^{2}(|x| \leq \frac{t}{2})}^{2} + \| v \|_{L_{x}^{p + 1}}^{p + 1} \lesssim_{u_{0}, u_{1}} t^{-2(1 - s_{c})}.
\end{equation}
Thus, we have good decay of the energy inside the light cone.
\end{remark}
\begin{proof}
Again as in  \cite{dodson2023sharp1}, let $\chi \in C_{0}^{\infty}(\mathbb{R}^{4})$ be a smooth, cutoff function, $\chi(x) = 1$ for $|x| \leq 1$ and $\chi(x) = 0$ for $|x| \geq 2$. Then, split the initial data,
\begin{equation}\label{2.4}
u_{0} = v_{0} + w_{0}, \qquad u_{1} = v_{1} + w_{1},
\end{equation}
where
\begin{equation}\label{2.5}
v_{0} = \chi(\frac{x}{R}) P_{\leq N} u_{0}, \qquad v_{1} = \chi(\frac{x}{R}) P_{\leq N} u_{1},
\end{equation}
for some $0 < R < \infty$ and $0 < N < \infty$. Here $P_{\leq N}$ is the standard Littlewood--Paley projection to frequencies $\leq N$. By the dominated convergence theorem, there exists some $N < \infty$ such that
\begin{equation}\label{2.6}
\| P_{> N} u_{0} \|_{\dot{H}^{s_{c}}} + \| P_{> N} u_{1} \|_{\dot{H}^{s_{c} - 1}} \leq \frac{\epsilon}{2}.
\end{equation}
It is convenient to rescale the initial data so that $N = 1$.\medskip

Decompose $(\ref{2.2})$ into a system of equations,
\begin{equation}\label{2.15}
\aligned
w_{tt} - \Delta w + |w|^{p - 1} w &= 0, \qquad w(0, x) = w_{0}, \qquad w_{t}(0, x) = w_{1}, \\
v_{tt} - \Delta v + F &= 0, \qquad F = |u|^{p - 1} u - |w|^{p - 1} w, \qquad v(0, x) = v_{0}, \qquad v_{t}(0, x) = v_{1}.
\endaligned
\end{equation}
By $(\ref{2.6})$, $(\ref{4.5})$, Theorems $\ref{t4.2}$, $\ref{t4.3}$ and Lemmas $\ref{l4.4}$ and $\ref{l4.5}$,
\begin{equation}\label{2.16.1}
\aligned
\| w \|_{L_{t}^{2} L_{x}^{r}(\mathbb{R} \times \mathbb{R}^{d})} + \| w \|_{L_{t,x}^{q}(\mathbb{R} \times \mathbb{R}^{d})} + \| |x|^{\frac{2}{p - 1}} w \|_{L_{t,x}^{\infty} (\mathbb{R} \times \mathbb{R}^{d})} + \| w \|_{L_{t}^{\infty} \dot{H}^{s_{c}}(\mathbb{R} \times \mathbb{R}^{d})} \lesssim \epsilon, \\ \qquad \text{where} \qquad \frac{1}{q} = \frac{2}{(d + 1)(p - 1)}, \qquad \frac{1}{r} = \frac{1}{d} (\frac{2}{p - 1} - \frac{1}{2}).
\endaligned
\end{equation}

Since $1 < p - 1 < 2$ when $d = 4, 5$,
\begin{equation}\label{2.17}
||u|^{p - 1} u - |w|^{p - 1} w - |v|^{p - 1} v| \lesssim |w| |v| (|v|^{p - 2} + |w|^{p - 2}).
\end{equation}
By the radial Sobolev embedding theorem,
\begin{equation}\label{2.19}
\| |x| (|u|^{p - 1} u - |v|^{p - 1} v - |w|^{p - 1} w) \|_{L_{x}^{2}} \lesssim \| |x|^{\frac{2}{p - 1}} w \|_{L_{x}^{\infty}}^{\frac{p - 1}{2}} \| w \|_{L_{x}^{p + 1}}^{\frac{3 - p}{2}} \| v \|_{L_{x}^{p + 1}}^{p - 1} + \| |x|^{\frac{2}{p - 1}} w \|_{L_{x}^{\infty}}^{\frac{p - 1}{2}} \| w \|_{L_{x}^{p + 1}}^{\frac{p - 1}{2}} \| v \|_{L_{x}^{p + 1}},
\end{equation}
and
\begin{equation}\label{2.20}
\| |t|^{\frac{2}{p - 1}} |w| \|_{L_{x}^{\infty}(|x| > \delta |t|)}^{\frac{p - 1}{2}} \lesssim \frac{1}{\delta} \| |x|^{\frac{2}{p - 1}} w \|_{L_{x}^{\infty}}^{\frac{p - 1}{2}},
\end{equation}
so
\begin{equation}\label{2.21}
\aligned
\| |t| ((|u|^{p - 1} u - |v|^{p - 1} v - |w|^{p - 1} w) \|_{L_{x}^{2}(|x| > \delta |t|)} \\ \lesssim \frac{1}{\delta} \| |x|^{\frac{2}{p - 1}} w \|_{L_{x}^{\infty}}^{\frac{p - 1}{2}} \| w \|_{L_{x}^{p + 1}}^{\frac{3 - p}{2}} \| v \|_{L_{x}^{p + 1}}^{p - 1} + \frac{1}{\delta} \| |x|^{\frac{2}{p - 1}} w \|_{L_{x}^{\infty}}^{\frac{p - 1}{2}} \| w \|_{L_{x}^{p + 1}}^{\frac{p - 1}{2}} \| v \|_{L_{x}^{p + 1}}.
\endaligned
\end{equation}
Since $\| v \|_{L_{x}^{p + 1}}^{p + 1} \lesssim \frac{\mathcal E(t)}{t^{2}}$, by $(\ref{2.16.1})$,
\begin{equation}\label{2.22}
\aligned
\frac{d}{dt} \mathcal E(t) = -t \int_{|x| \leq \delta |t|} [ (t + |x|) Lv + (d - 1)v ] (|u|^{p - 1} u - |w|^{p - 1} w - |v|^{p - 1} v) dx \\ - t \int_{|x| \leq \delta |t|} [(t - |x|) \underline{L}v + (d - 1)v] (|u|^{p - 1} u - |w|^{p - 1} w - |v|^{p - 1} v) dx \\ + \frac{\epsilon^{\frac{p - 1}{2}}}{\delta} \frac{\mathcal E(t)^{\frac{1}{2} + \frac{p - 1}{p + 1}}}{t^{\frac{2(p - 1)}{p + 1}}} \| w \|_{L_{x}^{p + 1}}^{\frac{3 - p}{2}} + \frac{\epsilon^{\frac{p - 1}{2}}}{\delta} \frac{\mathcal E(t)^{\frac{1}{2} + \frac{1}{p + 1}}}{t^{\frac{2}{p + 1}}} \| w \|_{L_{x}^{p + 1}}^{\frac{p - 1}{2}}.
\endaligned
\end{equation}

Following \cite{dodson2023sharp}, by the fundamental theorem of calculus, for $t > \frac{1}{\delta^{1/2}}$,
\begin{equation}\label{2.23}
\mathcal E(t) = \mathcal E(\delta^{1/2} t) + \int_{\delta^{1/2} t}^{t} \frac{d}{d \tau} \mathcal E(\tau) d\tau,
\end{equation}
and for $1 < t < \frac{1}{\delta^{1/2}}$,
\begin{equation}\label{2.23.1}
\mathcal E(t) = \mathcal E(1) + \int_{1}^{t} \frac{d}{d \tau} \mathcal E(\tau) d\tau.
\end{equation}
Now, by a change of variables,
\begin{equation}\label{2.24}
\aligned
\int_{\delta^{-1/2}}^{\infty} \frac{\mathcal E(\delta^{1/2} t)^{\frac{1}{2(1 - s_{c})}}}{t^{\frac{1}{1 - s_{c}}}} dt = \delta^{-\frac{1}{2} + \frac{1}{2(1 - s_{c})}} \int_{\delta^{-1/2}}^{\infty} \frac{1}{\delta^{\frac{1}{2(1 - s_{c})}} t^{\frac{1}{1 - s_{c}}}} \mathcal E(\delta^{1/2} t)^{\frac{1}{2(1 - s_{c})}} \delta^{1/2} dt \\ = \delta^{-\frac{1}{2} + \frac{1}{2(1 - s_{c})}} \int_{1}^{\infty} \frac{\mathcal E(t)^{\frac{1}{2(1 - s_{c})}}}{t^{\frac{1}{1 - s_{c}}}} dt,
\endaligned
\end{equation}
and
\begin{equation}\label{2.24.1}
\int_{1}^{\delta^{-1/2}} \frac{\mathcal E(1)^{\frac{1}{2(1 - s_{c})}}}{t^{\frac{1}{1 - s_{c}}}} dt \lesssim \mathcal E(1).
\end{equation}
Next, by $(\ref{5.15.7})$ and $(\ref{5.15.8})$, letting $t' = \sup \{ 1, \delta^{1/2} t \}$ to simplify notation,
\begin{equation}\label{2.25}
\aligned
\int_{1}^{\infty} \frac{1}{t^{\frac{1}{1 - s_{c}}}} [\int_{t'}^{t} \frac{\epsilon^{\frac{p - 1}{2}}}{\delta} \frac{\mathcal E(\tau)^{\frac{1}{2} + \frac{p - 1}{p + 1}}}{\tau^{\frac{2(p - 1)}{p + 1}}} \| w(\tau) \|_{L_{x}^{p + 1}}^{\frac{3 - p}{2}} d\tau + \int_{t'}^{t} \frac{\epsilon^{\frac{p - 1}{2}}}{\delta} \frac{\mathcal E(\tau)^{\frac{1}{2} + \frac{1}{p + 1}}}{\tau^{\frac{2}{p + 1}}} \| w(\tau) \|_{L_{x}^{p + 1}}^{\frac{p - 1}{2}} d\tau]^{\frac{1}{2(1 - s_{c})}} dt \\
\lesssim \frac{\epsilon^{\frac{p - 1}{2} \cdot \frac{1}{2(1 - s_{c})}}}{\delta^{\frac{1}{2(1 - s_{c})}}} \cdot  (\int_{1}^{\infty} \frac{\mathcal E(t)^{\frac{1}{2(1 - s_{c})}}}{t^{\frac{1}{1 - s_{c}}}} dt)^{\frac{1}{2} + \frac{p - 1}{p + 1}} (\int_{1}^{\infty} \| w(t) \|_{L_{x}^{p + 1}}^{\frac{p + 1}{2(1 - s_{c})}})^{\frac{3 - p}{2(p + 1)}} \\
+ \frac{\epsilon^{\frac{p - 1}{2} \cdot \frac{1}{2(1 - s_{c})}}}{\delta^{\frac{1}{2(1 - s_{c})}}} \cdot  (\int_{1}^{\infty} \frac{\mathcal E(t)^{\frac{1}{2(1 - s_{c})}}}{t^{\frac{1}{1 - s_{c}}}} dt)^{\frac{1}{2} + \frac{1}{p + 1}} (\int_{1}^{\infty} \| w(t) \|_{L_{x}^{p + 1}}^{\frac{p + 1}{2(1 - s_{c})}})^{\frac{p - 1}{2(p + 1)}}.
\endaligned
\end{equation}
By $(\ref{2.16.1})$, the right hand side of $(\ref{2.25})$ can be absorbed into
\begin{equation}
\int_{1}^{\infty} \frac{\mathcal E(t)^{\frac{1}{2(1 - s_{c})}}}{t^{\frac{1}{1 - s_{c}}}} dt.
\end{equation}
Therefore, if we could ignore the contribution of $-t \int_{|x| \leq \delta |t|} [ (t + |x|) Lv + (d - 1)v ] (|u|^{p - 1} u - |w|^{p - 1} w - |v|^{p - 1} v) dx$ and $- t \int_{|x| \leq \delta |t|} [(t - |x|) \underline{L}v + (d - 1)v] (|u|^{p - 1} u - |w|^{p - 1} w - |v|^{p - 1} v) dx$, the proof of Theorem $\ref{t2.1}$ would be complete.\medskip

The terms with $(t + |x|) Lv$ and $(t - |x|) \underline{L} v$ may be handled in exactly the same way, so using $(\ref{2.17})$ and $(\ref{2.22})$, it remains to compute
\begin{equation}\label{2.26}
 \int_{t'}^{t} \int_{|x| \leq \delta |t|} \tau [(\tau + |x|) Lv + (d - 1)v] |w| |v|^{p - 1} dx d\tau, \qquad \text{and} \qquad \int_{t'}^{t} \int_{|x| \leq \delta |t|} \tau [(\tau + |x|) Lv + (d - 1)v] |w|^{p - 1} |v| dx d\tau,
\end{equation}
separately.\medskip

It is convenient to replace $\tilde{\mathcal E}(t)$ by
\begin{equation}\label{2.27}
\tilde{\mathcal E}(t) = \sup_{0 \leq s \leq t} \mathcal E(t).
\end{equation}
Note that $\frac{d}{dt} \tilde{\mathcal E}(t)$ is bounded by the right hand side of the absolute value of $(\ref{2.22})$. We abuse notation and let $\mathcal E(t)$ denote $\tilde{\mathcal E}(t)$.\medskip 

Suppose for a moment
\begin{equation}\label{2.27.1}
\| \frac{1}{|x|^{3/2}} v \|_{L_{\tau, x}^{2}([t', t] \times |x| \leq \delta t)}^{2} \lesssim_{\delta} (\frac{\mathcal E(t)}{t^{2}} + \frac{\mathcal E(t)^{\frac{2}{p + 1}} t^{2s_{c} \cdot \frac{p - 1}{p + 1}}}{t^{2}}).
\end{equation}
By H{\"o}lder's inequality,
\begin{equation}\label{2.28}
\aligned
t \int_{t'}^{t} \int_{|x| \leq \mathcal \delta |t|} [(\tau + |x|) Lv + (d - 1)v] |w| |v|^{p - 1} dx d\tau \\ 
\lesssim t \| [(\tau + |x|) Lv + (d - 1)v \|_{L_{\tau}^{\infty} L_{x}^{2}} \| |x|^{\frac{2}{p - 1} - \frac{1}{2}} w \|_{L_{\tau}^{2} L_{x}^{\infty}} \| \frac{1}{|x|^{3/2}} v \|_{L_{\tau, x}^{2}(|x| \leq \mathcal \delta t)}^{\frac{3 - p}{p - 1}}  \| \frac{1}{|x|^{\frac{1}{p + 1}}} v \|_{L_{t,x}^{p + 1}(|x| \leq \mathcal \delta t)}^{p - 1 - \frac{3 - p}{p - 1}} \\
\lesssim_{\delta} t \mathcal E(t)^{1/2} \| |x|^{\frac{2}{p - 1} - \frac{1}{2}} w \|_{L_{\tau}^{2} L_{x}^{\infty}} (\frac{\mathcal E(t)}{t^{2}} + \frac{\mathcal E(t)^{\frac{2}{p + 1}} t^{2 s_{c} \cdot \frac{p - 1}{p + 1}}}{t^{2}})^{1/2}.
\endaligned
\end{equation}
Turning to the second term in $(\ref{2.26})$,
\begin{equation}\label{2.36}
\aligned
t \int_{t'}^{t} \int_{|x| \leq \mathcal \delta |t|} [(\tau + |x|) Lv + (d - 1)v] |w|^{p - 1} |v| dx d\tau \\ 
\lesssim t \| [(\tau + |x|) Lv + (d - 1)v \|_{L_{\tau}^{\infty} L_{x}^{2}} \| |x|^{\frac{2}{p - 1} - \frac{1}{2}} w \|_{L_{\tau}^{2} L_{x}^{\infty}} \| \frac{1}{|x|^{3/2}} v \|_{L_{\tau, x}^{2}(|x| \leq \mathcal \delta t)}  \| |x|^{\frac{2}{p - 1}} w \|_{L_{t,x}^{\infty}}^{p - 2} \\
\lesssim_{\delta} \epsilon^{p - 2} t \mathcal E(t)^{1/2} \| |x|^{\frac{2}{p - 1} - \frac{1}{2}} w \|_{L_{\tau}^{2} L_{x}^{\infty}} (\frac{\mathcal E(t)}{t^{2}} + \frac{\mathcal E(t)^{\frac{2}{p + 1}} t^{2 s_{c} \cdot \frac{p - 1}{p + 1}}}{t^{2}})^{1/2}.
\endaligned
\end{equation}

\begin{remark}
All time intervals are $[t', t]$ where $t' = \sup \{ 1, \delta^{1/2} t \}$. The last estimate will be proved later using Proposition $\ref{p5.5}$.
\end{remark}
Now then, since $\| |x|^{\frac{2}{p - 1} - \frac{1}{2}} w \|_{L_{t}^{2} L_{x}^{\infty}} \lesssim \epsilon \ll_{\delta} 1$,
\begin{equation}\label{2.29}
t \mathcal E(t)^{1/2} \frac{\mathcal E(t)^{1/2}}{t} \epsilon \ll \mathcal E(t),
\end{equation}
Therefore, plugging $(\ref{2.28})$ into $(\ref{2.22})$ and using $(\ref{2.24})$, $(\ref{2.24.1})$, and $(\ref{2.25})$,
\begin{equation}\label{2.30}
\aligned
\int_{1}^{\infty} \frac{1}{t^{\frac{1}{1 - s_{c}}}} \mathcal E(t)^{\frac{1}{2(1 - s_{c})}} dt \lesssim \mathcal E(1)^{\frac{1}{2(1 - s_{c})}} + 1 \\ + \int_{1}^{\infty} \frac{1}{t^{\frac{1}{1 - s_{c}}}} \mathcal E(t)^{(\frac{1}{2} + \frac{1}{p + 1}) \cdot \frac{1}{2(1 - s_{c})}} t^{2 s_{c}(\frac{1}{2} - \frac{1}{p + 1}) \cdot \frac{1}{2(1 - s_{c})}} \| |x|^{\frac{2}{p - 1} - \frac{1}{2}} w \|_{L_{t}^{2} L_{x}^{\infty}([t', t] \times \mathbb{R}^{d})}^{\frac{1}{2(1 - s_{c})}} dt.
\endaligned
\end{equation}
Then by the Cauchy--Schwarz inequality,
\begin{equation}\label{2.31}
\aligned
\int_{1}^{\infty} \frac{1}{t^{\frac{1}{1 - s_{c}}}} \mathcal E(t)^{(\frac{1}{2} + \frac{1}{p + 1}) \cdot \frac{1}{2(1 - s_{c})}} t^{2 s_{c}(\frac{1}{2} - \frac{1}{p + 1}) \cdot \frac{1}{2(1 - s_{c})}} \| |x|^{\frac{2}{p - 1} - \frac{1}{2}} w \|_{L_{t}^{2} L_{x}^{\infty}([t', t] \times \mathbb{R}^{d})}^{\frac{1}{2(1 - s_{c})}} dt \\
\lesssim (\int_{1}^{\infty} \frac{\mathcal E(t)^{\frac{1}{2(1 - s_{c})}}}{t^{\frac{1}{1 - s_{c}}}} dt)^{\frac{1}{2} + \frac{1}{p + 1}} (\int_{1}^{\infty} t^{\frac{s_{c}}{1 - s_{c}}} t^{-\frac{1}{1 - s_{c}}} \| |x|^{\frac{2}{p - 1} - \frac{1}{2}} w \|_{L_{t}^{2} L_{x}^{\infty}([t', t] \times \mathbb{R}^{d})}^{\frac{2(p + 1)}{p - 1}})^{\frac{1}{2} - \frac{1}{p + 1}} \\
\lesssim (\int_{1}^{\infty} \frac{\mathcal E(t)^{\frac{1}{2(1 - s_{c})}}}{t^{\frac{1}{1 - s_{c}}}} dt)^{\frac{1}{2} + \frac{1}{p + 1}} (\int_{1}^{\infty} \frac{1}{t} \| |x|^{\frac{2}{p - 1} - \frac{1}{2}} w \|_{L_{t}^{2} L_{x}^{\infty}([t', t] \times \mathbb{R}^{d})}^{\frac{2(p + 1)}{p - 1}} dt)^{\frac{1}{2} - \frac{1}{p + 1}} \\
\lesssim (\int_{1}^{\infty} \frac{\mathcal E(t)^{\frac{1}{2(1 - s_{c})}}}{t^{\frac{1}{1 - s_{c}}}} dt)^{\frac{1}{2} + \frac{1}{p + 1}} \| |x|^{\frac{2}{p - 1} - \frac{1}{2}} w \|_{L_{t}^{2} L_{x}^{\infty}(\mathbb{R} \times \mathbb{R}^{d})}.
\endaligned
\end{equation}
This proves
\begin{equation}\label{2.32}
\int_{1}^{\infty} \frac{\mathcal E(t)^{\frac{1}{2(1 - s_{c})}}}{t^{\frac{1}{1 - s_{c}}}} dt \lesssim 1 + \mathcal E(1)^{\frac{1}{2(1 - s_{c})}}.
\end{equation}
\end{proof}

It only remains to prove $(\ref{2.27.1})$.
\begin{proposition}\label{p2.2}
For any $T > 1$, $T' = \sup \{ 1, \delta^{1/2} T \}$,
\begin{equation}\label{2.37}
\aligned
\int_{T'}^{T} \int \chi(\frac{x}{\delta T}) [\frac{1}{|x|^{3}} v^{2} + \frac{1}{|x|} |v|^{p + 1}] dx dt \lesssim_{\delta} \frac{\mathcal E(T)}{T^{2}} + \frac{T^{2 s_{c} \cdot \frac{p - 1}{p + 1}} \mathcal E(T)^{\frac{2}{p + 1}}}{T^{2}}.
\endaligned
\end{equation}
\end{proposition}
\begin{proof}
This proposition follows from Proposition $\ref{p5.5}$. First,
\begin{equation}\label{2.33}
\| v(t) \|_{L^{p + 1}}^{p + 1} \lesssim \frac{\mathcal E(t)}{t^{2}}.
\end{equation}
Next, for $\delta \ll 1$ and $\delta^{1/2} T \leq t \leq T$, from $(\ref{5.15.3})$,
\begin{equation}\label{2.34}
\| \chi(\frac{x}{\delta t}) \nabla_{t,x} v(t) \|_{L^{2}}^{2} \lesssim \frac{\mathcal E(t)}{t^{2}} + \frac{1}{t^{2}} \| v(t) \|_{L_{x}^{2}(|x| \leq 2 \delta t)}^{2}.
\end{equation}
By H{\"o}lder's inequality,
\begin{equation}\label{2.34.1}
\frac{1}{t^{2}} \| v(t) \|_{L^{2}(|x| \leq 2 \delta t)}^{2} \lesssim \frac{1}{t^{2}} t^{d(1 - \frac{2}{p + 1})} \| v(t) \|_{L^{p + 1}}^{2} \lesssim \frac{t^{2 s_{c} \cdot \frac{p - 1}{p + 1}} \mathcal E(t)^{\frac{2}{p + 1}}}{t^{2}}.
\end{equation}

Next, we turn to the error term in $(\ref{2.6})$. Plugging in $(\ref{2.17})$ to $F$,
\begin{equation}\label{2.38}
\aligned
\int_{T'}^{T} \int \chi(\frac{x}{\delta T}) \frac{1}{|x|} [|v|^{p} |w| + |v|^{2} |w|^{p - 1}] dx dt \\ \lesssim \eta \int_{T'}^{T} \int \chi(\frac{x}{\delta T}) \frac{1}{|x|} |v|^{p + 1} dx dt + C(\eta) \int_{T'}^{T} \int \chi(\frac{x}{\delta T}) \frac{1}{|x|} |v|^{2} |w|^{p - 1} dx dt \\
\lesssim \eta \int_{T'}^{T} \int \chi(\frac{x}{\delta T}) \frac{1}{|x|} |v|^{p + 1} dx dt + C(\eta) (\int_{T'}^{T} \int \chi(\frac{x}{\delta T}) |v|^{2} \frac{1}{|x|^{3}} dx dt) \| |x|^{\frac{2}{p - 1}} w \|_{L_{t,x}^{\infty}}^{p - 1}.
\endaligned
\end{equation}
Taking $\eta \ll \delta \ll 1$ and $\epsilon \ll \eta$, we can absorb the two terms on the right hand side into the left hand side of $(\ref{5.38})$. Now then,
\begin{equation}\label{2.39}
\int_{T'}^{T} \int \chi(\frac{x}{\delta T}) |v_{r}| |F| dx dt \lesssim_{\delta} \frac{1}{T} \int_{T'}^{T} \int \chi(\frac{x}{\delta T}) |t v_{r} + (d - 1) v| |F| dx dt + \frac{1}{T} \int_{T'}^{T} \int \chi(\frac{x}{\delta T}) |v| |F| dx dt.
\end{equation}
Following the computations in $(\ref{2.28})$ and $(\ref{2.36})$,
\begin{equation}\label{2.40}
\aligned
\frac{1}{T} \int_{T'}^{T} \int \chi(\frac{x}{\delta T}) |t v_{r} + (d - 1) v| |F| dx dt \\ \lesssim_{\delta} \epsilon \frac{\mathcal E(T)^{1/2}}{T} (\int_{T'}^{T} \int \chi(\frac{x}{\delta T}) \frac{1}{|x|^{3}} |v|^{2} dx dt + \int_{T'}^{T} \int \chi(\frac{x}{\delta T}) \frac{1}{|x|} |v|^{p + 1} dx dt)^{1/2} \\
\lesssim_{\delta} \epsilon \frac{\mathcal E(T)}{T^{2}} + \epsilon (\int_{T'}^{T} \int \chi(\frac{x}{\delta t}) [\frac{1}{|x|^{3}} |v|^{2} + \frac{1}{|x|} |v|^{p + 1}] dx dt).
\endaligned
\end{equation}
Since $\epsilon \ll \delta \ll 1$, the second term in $(\ref{2.40})$ can be absorbed onto the left hand side, proving the proposition.
\end{proof}

\section{A modified small data argument}
Extending the previous argument to the $d \geq 6$ case when $\frac{1}{2} < s_{c} < 1$ is complicated due to the fact that $p - 1 < 1$.\medskip

If $p - 1 < 1$,
\begin{equation}\label{7.1}
|u|^{p - 1} u - |v|^{p - 1} v - |w|^{p - 1} w \lesssim \inf \{ |w| |v|^{p - 1}, |v| |w|^{p - 1} \}.
\end{equation}

As in the case of the conformal nonlinear wave equation, since $w$ is a solution to the small data problem, $v$ should usually be larger than $w$, and when it is not, we can put that part with the equation for $w$ at minimal cost. Instead split $u = v + w$, where
\begin{equation}\label{7.4}
\aligned
v_{tt} - \Delta v + (1 - \chi(\frac{u}{w})) |u|^{p - 1} u &= 0, \qquad v(0, x) = v_{0}, \qquad v_{t}(0, x) = v_{1}, \\
w_{tt} - \Delta w + \chi(\frac{u}{w}) |u|^{p - 1} u &= 0, \qquad w(0, x) = w_{0}, \qquad w_{t}(0, x) = w_{1}, \\
\endaligned
\end{equation}
$(v_{0}, v_{1})$ and $(w_{0}, w_{1})$ satisfy $(\ref{2.4})$--$(\ref{2.6})$, and $\chi \in C_{0}^{\infty}(\mathbb{R})$, $\chi(x) = 1$ for $|x| \leq 3$ and $\chi(x) = 0$ for $|x| > 6$.\medskip

It is here that we encounter the main technical difficulty arising from the case when $\frac{1}{2} < s_{c} < 1$. When $s_{c} = \frac{1}{2}$, observe that the small data result in $(\ref{4.7.1})$ depends only on the size of the solution, not on the size of its derivatives. When $\frac{1}{2} < s_{c} < 1$, $(\ref{4.7.1})$ depends on the size of the derivatives, and non-local derivatives at that. Thus, carrying over the computations in $(\ref{4.7.1})$ to the nonlinearity $\chi(\frac{u}{w}) |u|^{p - 1} u$ is probably very difficult.

Instead, we make use of the energy estimate. Interpolating the standard energy estimate,
\begin{equation}\label{7.4.1}
\| \int_{0}^{t} \frac{\sin((t - \tau) \sqrt{-\Delta})}{\sqrt{-\Delta}} F(\tau) d\tau \|_{\dot{H}^{1} \times L^{2}} \lesssim \| F \|_{L_{t}^{1} L_{x}^{2}},
\end{equation}
with Corollary $\ref{c4.7}$, for any $0 \leq \frac{1}{r} \leq 2$, $\frac{1}{r'} = 1 - \frac{1}{r}$,
\begin{equation}\label{7.4.2}
\| \int_{0}^{t} \frac{\sin((t - \tau) \sqrt{-\Delta})}{\sqrt{-\Delta}} F(\tau) d\tau \|_{\dot{H}^{1} \times L^{2}} \lesssim \sum_{k} 2^{\frac{k}{r}} \| F \|_{L_{t}^{r'} L_{x}^{2}(\mathbb{R} \times \{ x : 2^{k} \leq |x| \leq 2^{k + 1} \})}.
\end{equation}

\begin{theorem}[Small data result]\label{t7.1}
The initial value problem
\begin{equation}\label{7.5}
w_{tt} - \Delta w + \chi(\frac{u}{w}) |u|^{p - 1} u = 0, \qquad w(0, x) = w_{0}, \qquad w_{t}(0, x) = w_{1},
\end{equation}
is globally well-posed and scattering. Moreover,
\begin{equation}\label{7.6}
\| w \|_{L_{t}^{\infty} \dot{H}^{s_{c}}(\mathbb{R} \times \mathbb{R}^{d})} + \| w \|_{L_{t,x}^{\frac{(p - 1)(d + 1)}{2}}(\mathbb{R} \times \mathbb{R}^{d})} + \| w \|_{L_{t}^{2} L_{x}^{\frac{2d(p - 1)}{5 - p}}(\mathbb{R} \times \mathbb{R}^{d})} \lesssim \epsilon.
\end{equation}
Also,
\begin{equation}\label{7.7}
\| |x|^{\frac{2}{p - 1} - \frac{1}{2}} w \|_{L_{t}^{2} L_{x}^{\infty}(\mathbb{R} \times \mathbb{R}^{d})} \lesssim \epsilon.
\end{equation}
\end{theorem}
\begin{remark}
Unlike \cite{dodson2023sharp1} in the conformal wave equation case, the proof of Theorem $\ref{t7.1}$ relies on the radial symmetry for the modified small data result.
\end{remark}

\begin{proof}
First note that by standard approximation analysis and persistence of regularity, $u$ and $w$ are smooth, we can assume $\chi(\frac{u}{w^{(n)}})$ is well-defined.\medskip

Define the Picard iteration scheme
\begin{equation}\label{7.8}
w^{(0)}(t) = \cos(t \sqrt{-\Delta}) w_{0} + \frac{\sin(t \sqrt{-\Delta})}{\sqrt{-\Delta}} w_{1},
\end{equation}
and for $n \geq 1$,
\begin{equation}\label{7.9}
w^{(n)}(t) = w^{(0)}(t) - \int_{0}^{t} \frac{\sin((t - \tau) \sqrt{-\Delta}}{\sqrt{-\Delta}} \chi(\frac{u}{w^{(n - 1)}}) |u(\tau)|^{p - 1} u(\tau) d\tau.
\end{equation}
First, since $|\frac{u}{w^{(n - 1)}}| \leq 6$ on the support of $\chi(\frac{u}{w^{(n - 1)}})$, decompose
\begin{equation}\label{7.10}
\chi(\frac{u}{w^{(n - 1)}}) |u(\tau)|^{p - 1} u(\tau) = \chi(\frac{u}{w^{(n - 1)}}) |\frac{u}{w^{(n - 1)}}|^{p - 1} \frac{u}{w^{(n - 1)}} \cdot |w^{(n - 1)}|^{p - 1} w^{(n - 1)}.
\end{equation}
Decompose
\begin{equation}\label{7.11}
|w^{(n - 1)}|^{p - 1} w^{(n - 1)} = \sum_{j} |w^{(n - 1)}|^{p - 1} (P_{j} w^{(n - 1)}).
\end{equation}
Let $\frac{1}{r} = \frac{1}{2} + \frac{p - 1}{2}$. By Theorem $\ref{t4.3}$,
\begin{equation}\label{7.12}
\aligned
\| |x|^{1 - \frac{1}{r}} |w^{(n - 1)}|^{p - 1} (P_{j} w^{(n - 1)}) \|_{L_{t,x}^{2}(|x| \leq 2^{-j})} \lesssim 2^{-j(1 - \frac{1}{r})} 2^{-j \frac{d}{2}} \| |x|^{\frac{2}{p - 1} - \frac{1}{2}} w^{(n - 1)} \|_{L_{t}^{2} L_{x}^{\infty}}^{p - 1} \| P_{j} w^{(n - 1)} \|_{L_{t}^{2} L_{x}^{\infty}} \\
\lesssim 2^{j(1 - s_{c})} \| |x|^{\frac{2}{p - 1} - \frac{1}{2}} w^{(n - 1)} \|_{L_{t}^{2} L_{x}^{\infty}}^{p - 1} \| P_{j} w^{(n - 1)} \|_{L_{t}^{2} L_{x}^{\frac{2d}{d - 1 - 2 s_{c}}}}.
\endaligned
\end{equation}
Also,
\begin{equation}\label{7.13}
\aligned
\sum_{k > -j} \| |x|^{1 - \frac{1}{r}} |w^{(n - 1)}|^{p - 1} (P_{j} w^{(n -1)}) \|_{L_{t}^{r} L_{x}^{2}(|x| \sim 2^{k})} \\ \lesssim \sum_{k > -j}  2^{-k(1 - s_{c})} \| |x|^{\frac{2}{p - 1} - \frac{1}{2}} w^{(n - 1)} \|_{L_{t}^{2} L_{x}^{\infty}}^{p - 1} \| P_{j} w^{(n - 1)} \|_{L_{t}^{2} L_{x}^{\frac{2d}{d - 1 - 2s_{c}}}} \\
\lesssim 2^{j(1 - s_{c})} \| |x|^{\frac{2}{p - 1} - \frac{1}{2}} w^{(n - 1)} \|_{L_{t}^{2} L_{x}^{\infty}}^{p - 1} \| P_{j} w^{(n - 1)} \|_{L_{t}^{2} L_{x}^{\frac{2d}{d - 1 - 2s_{c}}}}.
\endaligned
\end{equation}
Therefore, interpolating Corollary $\ref{c4.7}$ with the standard energy estimate for the wave equation, 
\begin{equation}\label{7.13.1}
\aligned
\| \int_{0}^{t} \frac{\sin((t - \tau) \sqrt{-\Delta})}{\sqrt{-\Delta}} \chi(\frac{u}{w^{(n - 1)}}) |\frac{u}{w^{(n - 1)}}|^{p - 1} \frac{u}{w^{(n - 1)}} \cdot |w^{(n - 1)}|^{p - 1} (P_{j} w^{(n - 1)}) d\tau \|_{\dot{H}^{1} \times L^{2}} \\ \lesssim 2^{j(1 - s_{c})} \| |x|^{\frac{2}{p - 1} - \frac{1}{2}} w^{(n - 1)} \|_{L_{t}^{2} L_{x}^{\infty}}^{p - 1} \| P_{j} w^{(n - 1)} \|_{L_{t}^{2} L_{x}^{\frac{2d}{d - 1 - 2s_{c}}}}.
\endaligned
\end{equation}
Now, observe that from the radial Strichartz estimates, for any $d > 3$, there exists $\delta(d) > 0$ such that $L_{t}^{2} L_{x}^{\frac{2d}{d - 2 + 2 \delta}}$ is an admissible pair. Replacing $L_{t}^{2} L_{x}^{\frac{2d}{d - 1 - 2 s_{c}}}$ with $L_{t}^{2} L_{x}^{\frac{2d}{d - 2 + 2 \delta}}$ in $(\ref{7.12})$--$(\ref{7.13.1})$,
\begin{equation}\label{7.13.2}
\aligned
\| \int_{0}^{t} \frac{\sin((t - \tau) \sqrt{-\Delta})}{\sqrt{-\Delta}} \chi(\frac{u}{w^{(n - 1)}}) |\frac{u}{w^{(n - 1)}}|^{p - 1} \frac{u}{w^{(n - 1)}} \cdot |w^{(n - 1)}|^{p - 1} (P_{j} w^{(n - 1)}) d\tau \|_{\dot{H}^{1/2} \times \dot{H}^{-1/2}} \\ \lesssim 2^{j \delta} \| |x|^{\frac{2}{p - 1} - \frac{1}{2}} w^{(n - 1)} \|_{L_{t}^{2} L_{x}^{\infty}}^{p - 1} \| P_{j} w^{(n - 1)} \|_{L_{t}^{2} L_{x}^{\frac{2d}{d - 2 - 2\delta}}}.
\endaligned
\end{equation}
Interpolating $(\ref{7.13.1})$ and $(\ref{7.13.2})$, there exists some $\sigma(p, d) > 0$ such that
\begin{equation}
\aligned
\| P_{k} \int_{0}^{t} \frac{\sin((t - \tau) \sqrt{-\Delta})}{\sqrt{-\Delta}} \chi(\frac{u}{w^{(n - 1)}}) |\frac{u}{w^{(n - 1)}}|^{p - 1} \frac{u}{w^{(n - 1)}} \cdot |w^{(n - 1)}|^{p - 1} (P_{j} w^{(n - 1)}) d\tau \|_{\dot{H}^{s_{c}} \times \dot{H}^{s_{c} - 1}} \\
\lesssim 2^{-\sigma |j - k|} \| |x|^{\frac{2}{p - 1} - \frac{1}{2}} w^{(n - 1)} \|_{L_{t}^{2} L_{x}^{\infty}}^{p - 1} (\| P_{j} w^{(n - 1)} \|_{L_{t}^{2} L_{x}^{\frac{2d}{d - 1 - 2s_{c}}}} + 2^{j(\delta + s_{c} - \frac{1}{2})} \| P_{j} w^{(n - 1)} \|_{L_{t}^{2} L_{x}^{\frac{2d}{d - 2 - 2 \delta}}}).
\endaligned
\end{equation}

Therefore, summing using Young's inequality,
\begin{equation}\label{7.13.3}
\aligned
\| \int_{0}^{t} \frac{\sin((t - \tau) \sqrt{-\Delta})}{\sqrt{-\Delta}} \chi(\frac{u}{w^{(n - 1)}}) |\frac{u}{w^{(n - 1)}}|^{p - 1} \frac{u}{w^{(n - 1)}} \cdot |w^{(n - 1)}|^{p - 1} w^{(n - 1)} d\tau \|_{\dot{H}^{s_{c}} \times \dot{H}^{s_{c} - 1}} \\ \lesssim \| |x|^{\frac{2}{p - 1} - \frac{1}{2}} w^{(n - 1)} \|_{L_{t}^{2} L_{x}^{\infty}}^{p - 1} (\sum_{j} 2^{2j(\delta + s_{c} - \frac{1}{2})} \| P_{j} w^{(n - 1)} \|_{L_{t}^{2} L_{x}^{\frac{2d}{d - 2 - 2\delta}}}^{2} +  \| P_{j} w^{(n - 1)} \|_{L_{t}^{2} L_{x}^{\frac{2d}{d - 1 - 2 s_{c}}}}^{2})^{1/2}.
\endaligned
\end{equation}
Therefore, by the Christ--Kiselev lemma,
\begin{equation}\label{7.13.4}
\aligned
\| |x|^{\frac{2}{p - 1}} w^{(n)} \|_{L_{t, x}^{\infty}} + \| |x|^{\frac{2}{p - 1} - \frac{1}{2}} w^{(n)} \|_{L_{t}^{2} L_{x}^{\infty}} \\ + (\sum_{j} 2^{2j(\delta + s_{c} - \frac{1}{2})} \| P_{j} w^{(n)} \|_{L_{t}^{2} L_{x}^{\frac{2d}{d - 2 - 2\delta}}}^{2} + \| P_{j} w^{(n)} \|_{L_{t}^{2} L_{x}^{\frac{2d}{d - 1 - 2 s_{c}}}}^{2})^{1/2} \\
\lesssim \epsilon + \| |x|^{\frac{2}{p - 1}} w \|_{L_{t, x}^{\infty}}^{p} + \| |x|^{\frac{2}{p - 1} - \frac{1}{2}} w \|_{L_{t}^{2} L_{x}^{\infty}}^{p} \\
+ (\sum_{j} 2^{2j(\delta + s_{c} - \frac{1}{2})} \| P_{j} w^{(n - 1)} \|_{L_{t}^{2} L_{x}^{\frac{2d}{d - 2 - 2\delta}}}^{2} + \| P_{j} w^{(n - 1)} \|_{L_{t}^{2} L_{x}^{\frac{2d}{d - 1 - 2 s_{c}}}}^{2})^{p/2}.
\endaligned
\end{equation}
Therefore, the left hand side of $(\ref{7.13.4})$ is uniformly bounded by $\epsilon$ for all $n$.

Next,
\begin{equation}\label{7.12}
w^{(n + 1)}(t) - w^{(n)}(t) = \int_{0}^{t} \frac{\sin((t - \tau) \sqrt{-\Delta}}{\sqrt{-\Delta}} [\chi(\frac{u}{w^{(n)}}) - \chi(\frac{u}{w^{(n - 1)}})] |u(\tau)|^{p - 1} u(\tau) d\tau.
\end{equation}
We show
\begin{equation}\label{7.13}
[\chi(\frac{u}{w^{(n)}}) - \chi(\frac{u}{w^{(n - 1)}}] |u(\tau)|^{p - 1} u(\tau) \lesssim |w^{(n)} - w^{(n - 1)}|  (|w^{(n)}|^{p - 1} + |w^{(n - 1)}|^{p - 1}).
\end{equation}
Indeed, when $|w^{(n)} - w^{(n - 1)}| \gtrsim |w^{(n)}| + |w^{(n - 1)}|$,
\begin{equation}\label{7.14}
[\chi(\frac{u}{w^{(n)}}) - \chi(\frac{u}{w^{(n - 1)}})] |u(\tau)|^{p - 1} u(\tau) \lesssim |w^{(n)}|^{p} + |w^{(n - 1)}|^{p} \lesssim |w^{(n)} - w^{(n - 1)}| (|w^{(n)}|^{p - 1} + |w^{(n - 1)}|^{p - 1}).
\end{equation}
For $|w^{(n)} - w^{(n - 1)}| \ll |w^{(n)}| + |w^{(n - 1)}|$, by the fundamental theorem of calculus,
\begin{equation}\label{7.15}
\aligned
\chi(\frac{u}{w^{(n)}}) - \chi(\frac{u}{w^{(n - 1)}}) = \int_{0}^{1} \frac{d}{d\tau} \chi(\frac{u}{\tau w^{(n)} + (1 - \tau) w^{(n - 1)}}) d\tau \\ = -\int_{0}^{1} \chi'(\frac{u}{\tau w^{(n)} + (1 - \tau) w^{(n - 1)}}) \frac{u}{(\tau w^{(n)} + (1 - \tau) w^{(n - 1)})^{2}} \cdot (w^{(n)} - w^{(n - 1)}) d\tau.
\endaligned
\end{equation}
By the support properties of $\chi$,
\begin{equation}\label{7.16}
\chi'(\frac{u}{\tau w^{(n)} + (1 - \tau) w^{(n - 1)}}) \frac{u}{(\tau w^{(n)} + (1 - \tau) w^{(n - 1)})^{2}} \lesssim \frac{1}{u},
\end{equation}
so $(\ref{7.13})$ also holds. Therefore, for $\epsilon > 0$ sufficiently small, the contraction mapping principle implies that $w^{(n)}$ converges in $X$, where $X$ is the space given by the norm on the left hand side of $(\ref{7.13.4})$. By Theorem $\ref{t4.3}$ and Lemmas $\ref{l4.4}$ and $\ref{l4.5}$, $(\ref{7.6})$ and $(\ref{7.7})$ hold.
\end{proof}

\section{Scattering when $d > 5$}
Now we are ready to prove scattering when $d > 5$.
\begin{theorem}\label{t3.1}
If $d > 5$ and $u$ is a solution to the conformal wave equation,
\begin{equation}\label{3.1}
u_{tt} - \Delta u + |u|^{p - 1} u = 0, \qquad u(0,x) = u_{0} \in \dot{H}^{s_{c}}, \qquad u_{t}(0, x) = u_{1} \in \dot{H}^{s_{c} - 1}, \qquad u_{0}, u_{1} \qquad \text{radial},
\end{equation}
then $u$ is a global solution to $(\ref{3.1})$ and scatters, that is
\begin{equation}\label{3.2}
\| u \|_{L_{t}^{\frac{p + 1}{2(1 - s_{c})}} L_{x}^{p + 1}(\mathbb{R} \times \mathbb{R}^{d})} \leq C(u_{0}, u_{1}) < \infty.
\end{equation}
\end{theorem}
\begin{proof}
Using the partition $u = v + w$ in $(\ref{7.4})$, let
\begin{equation}\label{3.4}
F = (1 - \chi(\frac{u}{w})) |u|^{p - 1} u - |v|^{p - 1} v.
\end{equation}
As in $(\ref{2.17})$--$(\ref{2.22})$,
\begin{equation}\label{3.5}
\mathcal E(1) \lesssim R^{2} (\| v_{0} \|_{\dot{H}^{s_{c}}}^{2} + \| v_{1} \|_{\dot{H}^{s_{c} - 1}}^{2} + \| v_{0} \|_{\dot{H}^{s_{c}}}^{p + 1}),
\end{equation}
and
\begin{equation}\label{3.6}
\frac{d}{dt} \mathcal E(t) = -\langle (t + |x|) Lv + (d - 1)v, (t + |x|) F \rangle - \langle (t - |x|) \underline{L}v + (d - 1)v, (t - |x|) F \rangle.
\end{equation}
Since $1 - \chi(\frac{u}{w})$ is supported on the set $|u| \geq 3 |w|$, $|v| \gtrsim |w|$ on the support of $(1 - \chi(\frac{u}{w}))$. Therefore,
\begin{equation}\label{3.8}
F \lesssim |v|^{p}.
\end{equation}
Also,
\begin{equation}\label{3.7}
\aligned
F = (1 - \chi(\frac{u}{w})) |u|^{p - 1} u - |v|^{p - 1} v = [|u|^{p - 1} u - |v|^{p - 1} v] - \chi(\frac{u}{w}) |u|^{p - 1} u \\ \lesssim |w| (|v|^{p - 1} + |w|^{p - 1}) + |w|^{p} \lesssim |w| |v|^{p - 1}.
\endaligned
\end{equation}
We can handle the terms
\begin{equation}\label{3.9}
\aligned
\| |x| F \|_{L^{2}} + \| tF \|_{L^{2}(|x| \geq \delta |t|)},
\endaligned
\end{equation}
in a manner analogous to $(\ref{2.19})$--$(\ref{2.22})$.\medskip

By H{\"o}lder's inequality,
\begin{equation}\label{3.10}
\aligned
t \int_{t'}^{t} \int |(\tau + |x|) Lv + (d - 1) v| |F| dx d\tau \\
\lesssim t \| |x|^{-1/2} \{ (\tau + |x|) Lv + (d - 1)v \} \|_{L_{t,x}^{2}}^{2 - p} \| \frac{1}{|x|^{3/2}} v \|_{L_{t,x}^{2}} \| \{ (\tau + |x|) Lv + (d - 1)v \} \|_{L_{t}^{\infty} L_{x}^{2}}^{p - 1} \| |x|^{\frac{2}{p - 1} - \frac{1}{2}} w \|_{L_{t}^{2} L_{x}^{\infty}}^{p - 1} \\
\lesssim t \mathcal E(t)^{\frac{p - 1}{2}} \| |x|^{-1/2} \{ (\tau + |x|) Lv + (d - 1)v \} \|_{L_{t,x}^{2}}^{2 - p} \| \frac{1}{|x|^{3/2}} v \|_{L_{t,x}^{2}} \| |x|^{\frac{2}{p - 1} - \frac{1}{2}} w \|_{L_{t}^{2} L_{x}^{\infty}}^{p - 1}.
\endaligned
\end{equation}
Following Propositions $\ref{p2.2}$ and $\ref{p5.5}$, we have
\begin{equation}\label{3.10.1}
\| \frac{1}{|x|^{3/2}} v \|_{L_{t,x}^{2}(|x| \leq \delta t)}^{2} \lesssim \frac{\mathcal E(t)}{t^{2}} + \frac{t^{2s_{c} \cdot \frac{p - 1}{p + 1}} \mathcal E(t)^{\frac{2}{p + 1}}}{t^{2}}.
\end{equation}
If we also had
\begin{equation}\label{3.10.2}
\| |x|^{-1/2} \{ (\tau + |x|) Lv + (d - 1)v \} \|_{L_{t,x}^{2}}^{2} \lesssim \mathcal E(t) + t^{2 s_{c} \cdot \frac{p - 1}{p + 1}} \mathcal E(t)^{\frac{2}{p + 1}},
\end{equation}
then we could proceed as in the $d = 4$ and $d = 5$ case. However, Proposition $\ref{p5.5}$ doesn't quite give $(\ref{3.10.2})$, but only on an annulus $R \leq |x| \leq 2R$. So instead, split $|x| \leq \delta |t|$ into $|x| \leq \mathcal R$ and $\mathcal R \leq |x| \leq \delta |t|$.\medskip

Let
\begin{equation}\label{3.11}
\mathcal R = \inf \{ (\frac{\mathcal E(t)}{t^{2}} + \frac{t^{2 s_{c} \cdot \frac{p - 1}{p + 1}} \mathcal E(t)^{\frac{2}{p + 1}} }{t^{2}})^{-\frac{1}{2(1 - s_{c})}}, \delta |t| \}.
\end{equation}
Fix some $\sigma(p) > 0$.
\begin{equation}\label{3.12}
\aligned
t \int_{t'}^{t} \int_{\mathcal R \leq |x| \leq \delta |t|} |(\tau + |x|) Lv + (d - 1) v| |F| dx d\tau \\ \lesssim t \| (t + |x|) Lv + (d - 1) v \|_{L_{t}^{\infty} L_{x}^{2}}^{p - 1 - \sigma} \| |x|^{-1/2} |x|^{-\frac{(1 - s_{c})}{2 - p + \sigma} \sigma} \{ (t + |x|) Lv + (d - 1) v \} \|_{L_{t,x}^{2}(\mathcal R \leq |x| \leq \delta t)}^{2 - p + \sigma} \\ \times \| \frac{1}{|x|^{3/2}} v \|_{L_{t,x}^{2}}^{1 - \sigma} \| |x|^{\frac{2}{p - 1} - \frac{1}{2}} w \|_{L_{t}^{2} L_{x}^{\infty}}^{p - 1} \| |x|^{-s_{c}} w \|_{L_{t}^{\infty} L_{x}^{2}}^{\sigma} \\
\lesssim t \epsilon^{\sigma} \mathcal R^{-(1 - s_{c}) \sigma} \mathcal E(t)^{\frac{p - 1 - \sigma}{2}} (\mathcal E(t) + t^{2 s_{c} \cdot \frac{p - 1}{p + 1}} \mathcal E(t)^{\frac{2}{p + 1}})^{\frac{2 - p + \sigma}{2}}(\frac{\mathcal E(t)}{t^{2}} + \frac{t^{2 s_{c} \cdot \frac{p - 1}{p + 1}} \mathcal E(t)^{\frac{2}{p + 1}}}{t^{2}})^{\frac{1 - \sigma}{2}} \| |x|^{\frac{2}{p - 1} - \frac{1}{2}} w \|_{L_{t}^{2} L_{x}^{\infty}}^{p - 1} \\
\lesssim \epsilon^{\sigma} \mathcal E(t)^{\frac{p - 1 - \sigma}{2}} (\mathcal E(t) + t^{2 s_{c} \cdot \frac{p - 1}{p + 1}} \mathcal E(t)^{\frac{2}{p + 1}})^{\frac{3 - p + \sigma}{2}} \| |x|^{\frac{2}{p - 1} - \frac{1}{2}} w \|_{L_{t}^{2} L_{x}^{\infty}}^{p - 1}.
\endaligned
\end{equation}
\begin{remark}
Once again, the terms with $(t + |x|) Lv$ and $(t - |x|) \underline{L} v$ can be handled in exactly the same manner.
\end{remark}
Next,
\begin{equation}\label{3.15}
\aligned
t \int_{t'}^{t} \int_{|x| \leq \mathcal R} |(\tau + |x|) Lv + (d - 1) v| |F| dx d\tau \\ \lesssim t \| (t + |x|) Lv + (d - 1) v \|_{L_{t}^{\infty} L_{x}^{2}}^{p - 1 - \sigma} \| |x|^{-1/2} |x|^{\frac{(1 - s_{c})}{2 - p + \sigma} \cdot \sigma}  \{ (t + |x|) Lv + (d - 1) v \} \|_{L_{t,x}^{2}( |x| \leq \mathcal R)}^{2 - p + \sigma} \\ \times \| \frac{1}{|x|^{3/2}} v \|_{L_{t,x}^{2}} \| |x|^{\frac{d}{2} - 1} v \|_{L_{t,x}^{\infty}}^{\sigma} \| |x|^{\frac{2}{p - 1} - \frac{1}{2}} w \|_{L_{t}^{2} L_{x}^{\infty}}^{p - 1 - \sigma}  \\
\lesssim t \mathcal R^{(1 - s_{c}) \sigma} \mathcal E(t)^{\frac{p - 1 - \sigma}{2}} (\mathcal E(t) + t^{2 s_{c} \cdot \frac{p - 1}{p + 1}} \mathcal E(t)^{\frac{2}{p + 1}})^{\frac{2 - p + \sigma}{2}}(\frac{\mathcal E(t)}{t^{2}} + \frac{t^{2 s_{c} \cdot \frac{p - 1}{p + 1}} \mathcal E(t)^{\frac{2}{p + 1}}}{t^{2}})^{\frac{1 + \sigma}{2}} \| |x|^{\frac{2}{p - 1} - \frac{1}{2}} w \|_{L_{t}^{2} L_{x}^{\infty}}^{p - 1 - \sigma} \\
\lesssim \mathcal E(t)^{\frac{p - 1 - \sigma}{2}} (\mathcal E(t) + t^{2 s_{c} \cdot \frac{p - 1}{p + 1}} \mathcal E(t)^{\frac{2}{p + 1}})^{\frac{3 - p + \sigma}{2}} \| |x|^{\frac{2}{p - 1} - \frac{1}{2}} w \|_{L_{t}^{2} L_{x}^{\infty}}^{p - 1 - \sigma}.
\endaligned
\end{equation}
\begin{remark}
We use the radial Sobolev embedding theorem and $(\ref{2.34})$, $(\ref{2.34.1})$, to prove
\begin{equation}
\| |x|^{\frac{d}{2} - 1} v \|_{L_{t,x}^{\infty}(|x| \leq \delta |t|)}^{2} \lesssim \| \chi(\frac{x}{\delta |t|} v \|_{\dot{H}^{1}}^{2} \lesssim \frac{\mathcal E(t)}{t^{2}} + \frac{\mathcal E(t)^{\frac{2}{p + 1}} t^{2s_{c} \cdot \frac{p - 1}{p + 1}}}{t^{2}}.
\end{equation}
\end{remark}

Now then, following the computations in $(\ref{2.28})$--$(\ref{2.40})$, Theorem $\ref{t3.1}$ is proved, if we can prove

\begin{equation}\label{3.16.1}
\sup_{R > 0} R^{-1} \| \chi(\frac{x}{\delta t}) \nabla_{t,x} v \|_{L_{t,x}^{2}(|x| \leq R)}^{2} + \int_{t'}^{t} \int \chi(\frac{x}{\delta t}) [\frac{1}{|x|} |v|^{p + 1} + \frac{1}{|x|^{3}} |v|^{2}] dx dt \lesssim_{\delta} \frac{\mathcal E(t)}{t^{2}} + \frac{t^{2 s_{c} \cdot \frac{p - 1}{p + 1}} \mathcal E(t)^{\frac{2}{p + 1}}}{t^{2}}.
\end{equation}

\begin{proposition}\label{p3.2}
For $d > 4$, if $T > 1$, $T' = \sup \{ 1, \delta^{1/2} T \}$,
\begin{equation}\label{3.17}
\aligned
\sup_{R > 0} R^{-1} \int_{T'}^{T} \int_{|x| \leq R} \chi(\frac{x}{\delta T}) [|\nabla v|^{2} + v_{t}^{2} + |v|^{p + 1}] dx dt \\ + \int_{T'}^{T} \int \chi(\frac{x}{\delta T}) [\frac{1}{|x|^{3}} v^{2} + \frac{1}{|x|} |v|^{p + 1}] dx dt \lesssim_{\delta} \frac{\mathcal E(T)}{T^{2}} + \frac{T^{2 s_{c} \cdot \frac{p - 1}{p + 1}} \mathcal E(T)^{\frac{2}{p + 1}}}{T^{2}}.
\endaligned
\end{equation}
\end{proposition}
\begin{proof}
The computations in $(\ref{2.33})$--$(\ref{2.38})$ can be transferred over to Proposition $\ref{p3.2}$. Now then, following $(\ref{3.11})$ and $(\ref{3.12})$,
\begin{equation}\label{3.18}
\aligned
\int_{t'}^{t} \int \chi(\frac{x}{\delta t}) |v_{r}| |F| dx dt \lesssim (\sup_{R > 0} R^{-1} \int_{t'}^{t} \int \chi(\frac{x}{\delta t}) |\nabla_{t,x} v|^{2} dx d\tau)^{\frac{2 - p + \sigma}{2}} (\sup_{\tau \in [t', t]} \int \chi(\frac{x}{\delta t}) |\nabla_{t, x} v|^{2} dx)^{\frac{p - 1 - \sigma}{2}} \\
\times  \left[ (\int_{t'}^{t} \int \chi(\frac{x}{\delta t}) \frac{1}{|x|^{3}} v^{2})^{\frac{1 - \sigma}{2}} \mathcal R^{-(1 - s_{c}) \sigma} + (\int_{t'}^{t} \int \chi(\frac{x}{\delta t}) \frac{1}{|x|^{3}} v^{2})^{\frac{1 + \sigma}{2}} \mathcal R^{(1 - s_{c}) \sigma} \right].
\endaligned
\end{equation}
Absorbing terms into the space-time integrals into the left hand side and using the computations in $(\ref{2.33})$--$(\ref{2.38})$ proves the Proposition.
\end{proof}

Proposition $\ref{p3.2}$ completes the proof of Theorem $\ref{t3.1}$.

\end{proof}

\section{Profile decomposition}
\begin{proof}[Proof of Theorem $\ref{t1.1}$]
In light of Theorems $\ref{t2.1}$ and $\ref{t3.1}$, to prove Theorem $\ref{t1.1}$, it suffices to prove that if $(u_{n}^{0}, u_{n}^{1})$ is a sequence of initial data satisfying
\begin{equation}\label{6.1}
\| u_{0}^{n} \|_{\dot{H}^{s_{c}}} + \| u_{1}^{n} \|_{\dot{H}^{s_{c} - 1}} \leq A < \infty,
\end{equation}
then
\begin{equation}\label{6.2}
\sup_{n} \| u^{n} \|_{L_{t,x}^{q}(\mathbb{R} \times \mathbb{R}^{d})} < \infty, \qquad q = \frac{(d + 1)(p - 1)}{2},
\end{equation}
is uniformly bounded, where $u^{n}$ is the solution to $(\ref{1.1})$ with initial data $(u_{0}^{n}, u_{1}^{n})$.

Indeed, since $\dot{H}^{s_{c}} \times \dot{H}^{s_{c} - 1}$ is separable, we can take a dense sequence $(u_{0}^{n}, u_{1}^{n})$ in $(\ref{6.1})$. Passing to a subsequence where $(\ref{6.2})$ is increasing, it is enough to show that Theorems $\ref{t2.1}$, $\ref{t3.1}$, and a standard profile decomposition argument imply that $(\ref{6.2})$ is uniformly bounded for any $A$. By standard perturbative arguments, Theorem $\ref{t1.1}$ follows.
\begin{remark}
Observe that this argument does not give any idea how the function on the right hand side of $(\ref{1.3})$ depends on $A$.
\end{remark}
The argument proving $(\ref{6.2})$ is identical to the argument in \cite{dodson2018global} for the cubic wave equation, $(\ref{1.1})$ with $d = 3$, and uses the profile decomposition in \cite{ramos2012refinement}.

\begin{remark}
It is useful to use the notation $S(t)(f, g)$, which denotes the solution to the free wave equation with initial data $(f, g)$,
\begin{equation}\label{6.9}
S(t)(f, g) = \cos(t \sqrt{-\Delta}) f + \frac{\sin(t \sqrt{-\Delta})}{\sqrt{-\Delta}} g.
\end{equation}
\end{remark}

\begin{theorem}[Profile decomposition]\label{t6.1}
Suppose that there is a uniformly bounded, radially symmetric sequence
\begin{equation}\label{6.3}
\| u_{0}^{n} \|_{\dot{H}^{s_{c}}(\mathbb{R}^{d})} + \| u_{1}^{n} \|_{\dot{H}^{s_{c} - 1}(\mathbb{R}^{d})} \leq A < \infty.
\end{equation}
Then there exists a subsequence, also denoted $(u_{0}^{n}, u_{1}^{n}) \subset \dot{H}^{s_{c}} \times \dot{H}^{s_{c} - 1}$ such that for any $N < \infty$,
\begin{equation}\label{6.4}
S(t)(u_{0}^{n}, u_{1}^{n}) = \sum_{j = 1}^{N} \Gamma_{n}^{j} S(t)(\phi_{0}^{j}, \phi_{1}^{j}) + S(t)(R_{0, n}^{N}, R_{1,n}^{N}),
\end{equation}
with
\begin{equation}\label{6.5}
\lim_{N \rightarrow \infty} \limsup_{n \rightarrow \infty} \| S(t)(R_{0,n}^{N}, R_{1,n}^{N}) \|_{L_{t,x}^{q}(\mathbb{R} \times \mathbb{R}^{d})} = 0.
\end{equation}
Here, $\Gamma_{n}^{j} = (\lambda_{n}^{j}, t_{n}^{j})$ belongs to the group $(0, \infty) \times \mathbb{R}$, which acts by
\begin{equation}\label{6.6}
\Gamma_{n}^{j} F(t,x) = (\lambda_{n}^{j})^{\frac{2}{p - 1}} F(\lambda_{n}^{j} (t - t_{n}^{j}), \lambda_{n}^{j} x).
\end{equation}
The $\Gamma_{n}^{j}$ are pairwise orthogonal, that is, for every $j \neq k$,
\begin{equation}\label{6.7}
\lim_{n \rightarrow \infty} \frac{\lambda_{n}^{j}}{\lambda_{n}^{k}} + \frac{\lambda_{n}^{k}}{\lambda_{n}^{j}} + (\lambda_{n}^{j})^{1/2} (\lambda_{n}^{k})^{1/2} |t_{n}^{j} - t_{n}^{k}| = \infty.
\end{equation}
Furthermore, for every $N \geq 1$,
\begin{equation}\label{6.8}
\aligned
\| (u_{0, n}, u_{1, n}) \|_{\dot{H}^{s_{c}} \times \dot{H}^{s_{c} - 1}}^{2} = \sum_{j = 1}^{N} \| (\phi_{0}^{j}, \phi_{0}^{k}) \|_{\dot{H}^{s_{c}} \times \dot{H}^{s_{c} - 1}}^{2} + \| (R_{0, n}^{N}, R_{1, n}^{N}) \|_{\dot{H}^{s_{c}} \times \dot{H}^{s_{c} - 1}}^{2} + o_{n}(1).
\endaligned
\end{equation}
\end{theorem}
\begin{remark}
For any $N$, $\lim_{n \rightarrow \infty} o_{n}(1) = 0$.
\end{remark}

In the course of proving Theorem $\ref{t6.1}$, \cite{ramos2012refinement} proved
\begin{equation}\label{6.10}
S(\lambda_{n}^{j} t_{n}^{j})(\frac{1}{(\lambda_{n}^{j})^{\frac{2}{p - 1}}} u_{0}^{n}(\frac{x}{\lambda_{n}^{j}}), \frac{1}{(\lambda_{n}^{j})^{\frac{p + 1}{p - 1}}} u_{1}^{n}(\frac{x}{\lambda_{n}^{j}})) \rightharpoonup \phi_{0}^{j}(x),
\end{equation}
weakly in $\dot{H}^{s_{c}}(\mathbb{R}^{d})$, and
\begin{equation}\label{6.11}
\partial_{t}S(t + \lambda_{n}^{j} t_{n}^{j})(\frac{1}{(\lambda_{n}^{j})^{\frac{2}{p - 1}}} u_{0}^{n}(\frac{x}{\lambda_{n}^{j}}), \frac{1}{(\lambda_{n}^{j})^{\frac{p + 1}{p - 1}}} u_{1}^{n}(\frac{x}{\lambda_{n}^{j}}))|_{t = 0} \rightharpoonup \phi_{1}^{j}(x)
\end{equation}
weakly in $\dot{H}^{s_{c} - 1}(\mathbb{R}^{d})$.\medskip

Suppose that for some $j$, $\lambda_{n}^{j} t_{n}^{j}$ is uniformly bounded. Then after passing to a subsequence, $\lambda_{n}^{j} t_{n}^{j}$ converges to some $t^{j}$. Changing $(\phi_{0}^{j}, \phi_{1}^{j})$ to $(S(-t^{j})(\phi_{0}^{j}, \phi_{1}^{j}), \partial_{t} S(t - t^{j})(\phi_{0}^{j}, \phi_{1}^{j})|_{t = 0})$ and absorbing the error into $(R_{0, n}^{N}, R_{1, n}^{N})$,
\begin{equation}\label{6.12}
\frac{1}{(\lambda_{n}^{j})^{\frac{2}{p - 1}}} u_{0}^{n}(\frac{x}{\lambda_{n}^{j}}) \rightharpoonup \phi_{0}^{j}(x), \qquad \text{weakly in} \qquad \dot{H}^{s_{c}},
\end{equation}
and
\begin{equation}\label{6.13}
\partial_{t}S(t)(\frac{1}{(\lambda_{n}^{j})^{\frac{2}{p - 1}}} u_{0}^{n}(\frac{x}{\lambda_{n}^{j}}), \frac{1}{(\lambda_{n}^{j})^{\frac{p + 1}{p - 1}}} u_{1}^{n}(\frac{x}{\lambda_{n}^{j}}))|_{t = 0} \rightharpoonup \phi_{1}^{j}(x), \qquad \text{weakly in} \qquad \dot{H}^{s_{c} - 1}.
\end{equation}
If $u^{(j)}$ is the solution to $(\ref{1.1})$ with initial data $(\phi_{0}^{j}, \phi_{1}^{j})$, then
\begin{equation}\label{6.14}
\| u^{(j)} \|_{L_{t,x}^{q}(\mathbb{R} \times \mathbb{R}^{d})} \leq M_{j}.
\end{equation}

Next, suppose that after passing to a subsequence, $\lambda_{n}^{j} t_{n}^{j} \nearrow +\infty$. In this case, for any $(\phi_{0}^{j}, \phi_{1}^{j}) \in \dot{H}^{s_{c}} \times \dot{H}^{s_{c} - 1}$, there exists a solution $u^{(j)}$ to $(\ref{1.1})$ that is globally well-posed and scattering, and furthermore, that $u$ scatters to $S(t)(\phi_{0}^{j}, \phi_{1}^{j})$ as $t \searrow -\infty$.
\begin{equation}\label{6.15}
\lim_{t \rightarrow -\infty} \| u^{(j)}(t) - S(t)(\phi_{0}^{j}, \phi_{1}^{j}) \|_{\dot{H}^{s_{c}} \times \dot{H}^{s_{c} - 1}} = 0.
\end{equation}
Indeed, by Strichartz estimates, the dominated convergence theorem, and the small data arguments in Theorem $\ref{t4.2}$, for some $T_{j} < \infty$ sufficiently large, $(\ref{1.1})$ has a solution $u$ on $(-\infty, -T]$ such that
\begin{equation}\label{6.16}
\aligned
\| |\nabla|^{s_{c} - \frac{1}{2}} u^{(j)} \|_{L_{t,x}^{\frac{2(d + 1)}{d - 1}}((-\infty, -T_{j}] \times \mathbb{R}^{d})} + \| u^{(j)} \|_{L_{t,x}^{q}((-\infty, -T_{j}] \times \mathbb{R}^{d})} \lesssim \epsilon_{0}(d), \\ (u^{(j)}(-T_{j}, x), u_{t}^{(j)}(-T_{j}, x)) = S(-T_{j})(\phi_{0}^{j}, \phi_{1}^{j}),
\endaligned
\end{equation}
where $\epsilon_{0}(d) > 0$ is sufficiently small. Also by Strichartz estimates and small data arguments,
\begin{equation}\label{6.17}
\lim_{t \rightarrow +\infty} \| S(t)(u^{(j)}(-t), u_{t}^{(j)}(-t)) - (\phi_{0}, \phi_{1}) \|_{\dot{H}^{s_{c}} \times \dot{H}^{s_{c} - 1}} \lesssim \epsilon^{q}.
\end{equation}
Then by the inverse function theorem, there exists some $(u_{0}^{(j)}(-T_{j}), u_{1}^{(j)}(-T_{j}))$ such that $(\ref{1.1})$ has a solution that scatters backward in time to $S(t)(\phi_{0}^{j}, \phi_{1}^{j})$. Since $u_{0}^{(j)}(-T_{j}) \in \dot{H}^{s_{c}}$ and $u_{1}^{(j)}(-T_{j}) \in \dot{H}^{s_{c} - 1}$, $(\ref{1.1})$ has a solution that scatters forward and backward in time,
\begin{equation}\label{6.18}
\| u^{(j)} \|_{L_{t,x}^{q}(\mathbb{R} \times \mathbb{R}^{d})} \leq M_{j} < \infty,
\end{equation}
and $u^{(j)}(-T_{j}, x) = u_{0}^{(j)}(-T_{j}, x)$, $u_{t}^{(j)}(-T_{j}, x) = u_{1}^{(j)}(-T_{j}, x)$. Therefore,
\begin{equation}\label{6.19}
S(-t_{n}^{j})((\lambda_{n}^{j})^{\frac{2}{p - 1}} \phi_{0}^{j}(\lambda_{n}^{j} x), (\lambda_{n}^{j})^{\frac{p + 1}{p - 1}} \phi_{1}^{j}(\lambda_{n}^{j} x))
\end{equation}
converges strongly to
\begin{equation}\label{6.20}
((\lambda_{n}^{j})^{\frac{2}{p - 1}} u^{(j)}(-\lambda_{n}^{j} t_{n}^{j}, \lambda_{n}^{j} x), (\lambda_{n}^{j})^{\frac{p + 1}{p - 1}} u_{t}^{(j)}(-\lambda_{n}^{j} t_{n}^{j}, \lambda_{n}^{j} x))
\end{equation}
in $\dot{H}^{s_{c}} \times \dot{H}^{s_{c} - 1}$, where $u^{j}$ is the solution to $(\ref{1.1})$ that scatters backward in time to $S(t)(\phi_{0}^{j}, \phi_{1}^{j})$, and the remainder may be absorbed into $(R_{0, n}^{N}, R_{1, n}^{N})$. 


The proof for $\lambda_{n}^{j} t_{n}^{j} \searrow -\infty$ is similar.\medskip

By $(\ref{6.8})$, there are only finitely many $j$, say $J$, such that $\| \phi_{0}^{j} \|_{\dot{H}^{s_{c}}} + \| \phi_{1}^{j} \|_{\dot{H}^{s_{c} - 1}} > \epsilon_{0}(d)$. For all other $j$, small data arguments imply
\begin{equation}\label{6.22}
\| u^{(j)} \|_{L_{t,x}^{q}(\mathbb{R} \times \mathbb{R}^{d})} \lesssim \| \phi_{0}^{j} \|_{\dot{H}^{s_{c}}} + \| \phi_{1}^{j} \|_{\dot{H}^{s_{c} - 1}}.
\end{equation}
Then by the decoupling property $(\ref{6.7})$, $(\ref{6.8})$, $(\ref{6.12})$, $(\ref{6.22})$, and Theorem $\ref{t4.2}$,
\begin{equation}\label{6.23}
\limsup_{n \nearrow \infty} \| u^{n} \|_{L_{t,x}^{q}(\mathbb{R} \times \mathbb{R}^{d})}^{2} \lesssim \sum_{j} \| u^{(j)} \|_{L_{t,x}^{q}(\mathbb{R} \times \mathbb{R}^{d})}^{2} \lesssim \sum_{j = 1}^{J} M_{j}^{2} + A^{2} < \infty.
\end{equation}
This proves that $(\ref{6.2})$ holds.
\end{proof}

\section*{Acknowledgements}
During the time of writing this paper, the author was partially supported by NSF Grant DMS-$2153750$. The author is also grateful to Andrew Lawrie and Walter Strauss for some helpful conversations at MIT on subcritical nonlinear wave equations.

\bibliography{biblio}
\bibliographystyle{alpha}

\end{document}